\documentclass[reqno, 12pt, reqno]{amsart}

\usepackage{amsfonts, amsthm, amsmath, amssymb}
\usepackage{hyperref}
\hypersetup{colorlinks=false}
\usepackage{comment}

\usepackage[margin=3.3cm]{geometry}

\RequirePackage{mathrsfs}
\let\mathcal\mathscr

\numberwithin{equation}{section}

\newtheorem{theorem}{Theorem}[section]

\newtheorem{lemma}[theorem]{Lemma}

\newtheorem{corollary}[theorem]{Corollary}
\newtheorem{proposition}[theorem]{Proposition}

\theoremstyle{definition}

\newtheorem*{remark*}{Remark}

\newcommand{\BA}{{\mathbb {A}}}

\newcommand{\BC}{{\mathbb {C}}}

\newcommand{\BF}{{\mathbb {F}}}

\newcommand{\BL}{{\mathbb {L}}}

\newcommand{\BN}{{\mathbb {N}}}

\newcommand{\BP}{{\mathbb {P}}}
\newcommand{\BQ}{{\mathbb {Q}}}
\newcommand{\BR}{{\mathbb {R}}}

\newcommand{\BZ}{{\mathbb {Z}}}
\newcommand{\CA}{{\mathcal {A}}}

\newcommand{\CD}{{\mathcal {D}}}

\newcommand{\CF}{{\mathcal {F}}}
\newcommand{\CG}{{\mathcal {G}}}
\newcommand{\CH}{{\mathcal {H}}}
\newcommand{\CI}{{\mathcal {I}}}
\newcommand{\CJ}{{\mathcal {J}}}
\newcommand{\CK}{{\mathcal {K}}}

\newcommand{\CN}{{\mathcal {N}}}
\newcommand{\CO}{{\mathcal {O}}}
\newcommand{\CP}{{\mathcal {P}}}

\newcommand{\CR}{{\mathcal {R}}}
\newcommand{\CS}{{\mathcal {S}}}
\newcommand{\CT}{{\mathcal {T}}}
\newcommand{\CU}{{\mathcal {U}}}
\newcommand{\CV}{{\mathcal {V}}}
\newcommand{\CW}{{\mathcal {W}}}
\newcommand{\CX}{{\mathcal {X}}}

\newcommand{\bb}{{\mathbf{b}}}

\DeclareMathOperator{\lcm}{lcm}

\renewcommand{\phi}{\varphi}
\renewcommand{\rho}{\varrho}
\renewcommand{\epsilon}{\varepsilon}

\newcommand{\bx}{\boldsymbol{x}}
\newcommand{\by}{\boldsymbol{y}}

\newcommand{\bt}{\boldsymbol{t}}

\newcommand{\bc}{\boldsymbol{c}}

\newcommand{\bv}{\boldsymbol{v}}

\newcommand{\blambda}{\boldsymbol{\lambda}}

\newcommand{\bGamma}{\boldsymbol{\Gamma}}
\newcommand{\bsigma}{\boldsymbol{\sigma}}

\newcommand{\balpha}{\boldsymbol{\alpha}}
\newcommand{\bbeta}{\boldsymbol{\beta}}
\newcommand{\bxi}{\boldsymbol{\xi}}

\newcommand{\e}{\textup{e}}

\title[Ternary quadratic forms I]{Quantitative strong approximation for ternary quadratic forms I}
\author{Zhizhong Huang}
\date{Version October 2023, Revised on October 2024}

\address{Institute of Mathematics, Academy of Mathematics and Systems Science, Chinese Academy of Sciences, Beijing, 100190, China}
\email{zhizhong.huang@yahoo.com}

\begin{document}
	\begin{abstract}
		We derive asymptotic formulas with a secondary term for the (smoothly weighted) count of number of integer solutions of height $\leqslant B$ with local conditions to the equation $F(x_1,x_2,x_3)=m$, where $F$ is a non-degenerate indefinite ternary integral  quadratic form, and $m$ is a non-zero integer satisfying $-m\Delta_F=\square$ which can grow like $O(B^{2-\theta})$ for some fixed $\theta>0$. Our approach is based on the $\delta$-variant of the Hardy--Littlewood circle method developed by Heath-Brown. 
	\end{abstract}
	\maketitle
	\tableofcontents
\section{Introduction}
Let $F(\bx)\in\BZ[x_1,x_2,x_3]$ be a ternary integral  quadratic form and let $m$ be a non-zero integer, such that $(F=m)$ defines a smooth affine quadratic surface $W_m$ in $\BA^3$. We assume that $F$ is indefinite. 
The fact that the Brauer--Manin obstruction to strong approximation of rational points is the only one is well-understood, thanks to work \cite{CT-XuCompositio,CT-Xu}.
The problem of counting integral points on such varieties (more generally on symmetric varieties) has been addressed in many works. See notably \cite{Duke-Rudnick-Sarnak,EM}. 

Let $\Delta_F$ denote the discriminant of $F$. Depending on whether $-m\Delta_F$ is a perfect square or not, the order of growth of integral points features very different behaviour.
The case where $-m\Delta_F\neq\square$ is well-studied, and it is shown that the number of integral points of bounded height $B$ say, grows like $cB$, where the constant $c$ involves the Tamagawa measure of the corresponding adelic open set defined by a prescribed local condition (strictly speaking, intersecting the loci where orbits contain rational points). See \cite{Borovoi-Rudnick}. 

The current article focuses on the case $-m\Delta_F=\square$. While it is known that there is no Brauer--Manin obstruction, the infinite product of non-archimedean local densities (the so-called ``singular series'') diverges to $\infty$. We next describe our result concerning counting integral points with arbitrary local conditions, followed by a historical account.

\subsection*{Principal results}
Let $w:\BR^3\to\BR$ be an infinitely differentiable weight function, and we define the \emph{weighted singular integral} (with respect to the quadratic form $F$) to be \begin{equation}\label{eq:CIw}
	\CI(w):=\int_{\BR^3\times\BR}w(\bt)\e(\theta F(\bt))\operatorname{d}\bt\operatorname{d}\theta. 
\end{equation}
If $w$ is positive, our assumption that $F$ is indefinite (i.e. isotropic over the real numbers) guarantees that $\CI(w)>0$.

Let $\CW_m$ be the $\BZ$-scheme defined by $(F=m) \subset\BA^3_{\BZ}$. We assume that the adelic space $$\CW_m(\widehat{\BZ}):=\prod_{p<\infty}\CW_m(\BZ_p)$$ (with respect to the integral model $\CW_m$) is non-empty.  Let $L\in\BN$ an and let $\bGamma_m\in \CW_m(\BZ/L\BZ)$. 
For every prime $p$, we define the \emph{$p$-adic local density} (with respect to the congruence condition $(L,\bGamma_m)$) to be $$\sigma_{p}(\CW_m;L,\bGamma_m):=\lim_{k\to\infty}\frac{\#\{\bv\in \CW_m(\BZ/p^k\BZ):\bv\equiv\bGamma_m\bmod p^{\operatorname{ord}_p(L)}\}}{p^{2k}},$$ and we define \emph{the ``modified'' singular series} (with the convergence factors $\left(1-p^{-1}\right)_p$) to be \begin{equation}\label{eq:singser}
	\widehat{\mathfrak{G}}(\CW_m;L,\bGamma_m):=\prod_{p<\infty}\left(1-\frac{1}{p}\right)\sigma_p(\CW_m;L,\bGamma_m).
\end{equation}

 We define the counting function
 \begin{equation}
 	\CN_m(w;L,\bGamma_m;B):=\sum_{\substack{\bx\in\BZ^3,F(\bx)=m\\ \bx\equiv \bGamma_m\bmod L}}w\left(\frac{\bx}{B}\right).
 \end{equation}
 Throughout this article, the weight function $w$, the form $F$ as well as the integer $L$ are fixed.
 \begin{theorem}\label{thm:mainterm}
 	As $B\to\infty$, assume that $-m\Delta_F=\square$ and
 	\begin{equation}\label{eq:mB2theta}
 		m=O(B^{2-\theta}) \text{ for a certain fixed } \theta>0.
 	\end{equation}
 	Then $$\CN_m(w;L,\bGamma_m;B)= \CI(w)\widehat{\mathfrak{G}}(\CW_m;L,\bGamma_m) B\log B+O\left(B(\log B)^{\frac{1983}{1984}}\right).$$
 	Moreover, $$\frac{1}{\log\log B}\ll \widehat{\mathfrak{G}}(\CW_m;L,\bGamma_m)\ll 1.$$
 \end{theorem}
 
 Theorem \ref{thm:mainterm} is a quantitative form of the strong approximation property for $\CW_m$, which gives a leading constant differing slightly from those obtained in \cite{PART1,PART2}.
 Firstly, the constant $m$ is absent from the singular integral $\CI(w)$. After \cite{PART1}, 
 $\CI(w)$ features the (weighted) real volume of the projective curve defined by the homogeneous equation $(F=0)$ in $\BP^3$. Secondly, the order of growth is $B\log B$, and the appearance of $\log B$ coincides with the speed of divergency of the usual product of local densities, so that inserting the convergence factors $(1-\frac{1}{p})_p$ ensures the  absolute convergence of $\widehat{\mathfrak{G}}(\CW_m;L,\bGamma_m)$.
 In the Appendix, we show that such phenomena
 are in accordance with the framework of Chambert-Loir and Tschinkel \cite{CL-T} on the distribution of integral points. From this perspective, Theorem \ref{thm:mainterm} shows that integral points are equidistributed in $\CW_m(\widehat{\BZ})$.

 Our next theorem shows the existence of a secondary term of order $B$, which confirms a suggestion of Rudnick (as pointed out in \cite[p. 918]{Oh-Shah}).   Let $\omega$ be the function which counts the number of prime divisors.
\begin{theorem}\label{thm:mainsecondary}
	Under the assumption of Theorem \ref{thm:mainterm}, assume moreover that 
\begin{equation}\label{eq:omegam}
	\omega(m)=O(1).
\end{equation}
	 Then there exists $a=a(w;m;L,\bGamma_m)\in\BC$ such that, 
	$$\CN_m(w;L,\bGamma_m;B)=\CI(w)\widehat{\mathfrak{G}}(\CW_m;L,\bGamma_m) B\log B+a B+O\left(\frac{B}{(\log B)^{\frac{1}{4}(1-\frac{\sqrt{2}}{2})}}\right).$$ Moreover, we have $a=O(1)$ and $\widehat{\mathfrak{G}}(\CW_m;L,\bGamma_m)\asymp 1$.
\end{theorem}

If we view $m$ as fixed, such asymptotic formulas can be proven using different methods (see below). On the other hand, it is evident that there can be no points if $m\gg B^{2}$. The significant feature of our result is that we can allow $m$ to grow ``as fast as possible'' like \eqref{eq:mB2theta} in the $B$-aspect. We refer to \cite[Fig. 1]{K-K2} for numerics supporting the condition \eqref{eq:mB2theta}.
 In both theorems our method actually provides an error term which saves a small power of $\log B$ (compared to $B\log B$ and $B$ respectively).

\subsection*{Brief history}
 The earliest single example is discussed in \cite{Duke-Rudnick-Sarnak} (see also \cite{Borovoi-Rudnick}).  Later, the work of Oh--Shah \cite{Oh-Shah} achieves an asymptotic formula for general $F$, based on methods from homogeneous dynamics. Recently this problem (over general number fields) is revisited in the work of Xu--Zhang \cite{Xu-Zhang}.  Asymptotic formulas with a secondary term are established by the work of Kelmer--Kontorovich \cite{K-K} based on spectral methods\footnote{It is pointed out in \cite{K-K} that Oh and Shah also obtain an asymptotic formula of such a form in some unpublished preprint \cite{Oh-Shah2} which supersedes \cite{Oh-Shah}}. (See also the work \cite{HKKL} where the form $F$ is more specific.) We emphasize that in the work above $m$ is considered fixed.

For growing $m$, in \cite{K-K2}  it is shown for the specific form $x_2^2-4x_1x_3$ that the growth order $m=O(B^{\frac{3}{2}-\theta})$ is allowable, conditional on the Ramanujan conjecture. See also the work of Friedlander--Iwaniec \cite{F-I} of similar vein concerning ``small integer solutions'' of the particular equation $x_1^2+x_2^2-x_3^2=m$, where $m$ ranges over fundamental discriminants and can grow like \eqref{eq:mB2theta}. 

\subsection*{Methods and further comments}
Traditionally, the circle method has been applied with success when the quadratic form has more than four variables. It is Heath-Brown who develops his $\delta$-version of the circle method in \cite{H-Bdelta} (the predecessor goes back to work of Duke--Friedlander--Iwaniec \cite{D-F-I}), and first applies to varieties defined by \emph{homogeneous} ternary forms towards quantitative Hasse principle. His arguments rely crucially on homogeneity, which guarantees, amongst other things, the multiplicativity of the quadratic exponential sums arising naturally after the execution of Poisson summation.

In this article and the subsequent ones \cite{HuangII,HuangIII} we undertake systematic study of quantitative strong approximation results for ternary quadratic forms. This appears to be the first time (to our knowledge) that circle method is employed to inhomogeneous equations of the form defining $W_m$. Our approach is based on \cite{H-Bdelta}, and is inspired by our past investigation \cite{PART1,PART2}. The main technical challenge lies in dealing with the arithmetic part involving exponential sums, as when averaging over non-zero Poisson variables $\bc$, the associated quadratic exponential sums $\widehat{S}_q(\bc)$ are not multiplicative.  In this ternary case one is led to estimating average of certain Salié sums, for which the Weil bound is insufficient.
Depending on whether $-F^*(\bc)$ is a perfect square or not, eventually one is faced with the distribution of congruent residues of certain quadratic polynomials, and for which we draw ingredients from the classical work of Hooley \cite{Hooley} and the recent one due to Dartyge--Martin \cite{Dartyge-Martin}. While adapting these ingredients, more demanding control of quadratic exponential sums over ``bad moduli'' is also needed, and one has to go beyond the relatively easy square-root cancellation established by e.g. \cite[Lemma 25]{H-Bdelta}.
For Theorem \ref{thm:mainsecondary}, it might be possible to adapt some ideas from e.g. \cite{HooleyActa,D-F-I2} to achieve a power-saving error term, with some extra work.

Despite the treatment of the oscillatory integrals in the current article being relatively standard, in our subsequent investigation \cite{HuangII} towards Linnik-type problem this becomes an issue, for which we work out refined arguments there. 

Our result adds to the (rare) literature that the circle method is capable of producing asymptotic formulas whose order of growth is not merely a power of $B$. Also the real part (i.e. real density) and the arithmetic part (i.e. product of local densities) both arise very naturally. 
Moreover, the flexibility of allowing $m$ to grow like \eqref{eq:mB2theta} does not seem easily accessible via previously mentioned work, even assuming standard conjectures. See e.g. \cite[Remark 1.7]{K-K2}. Our approach applies without difficulty to the case $-m\Delta_F\neq\square$, subject to appropriate growth condition on its square-free part. 

The constant $a$ of the secondary term encodes the contribution from a mixture of certain oscillatory integrals and ``twisted quadratic exponential sums'' regarding the Poisson variables $\bc\in\BZ^3\setminus\boldsymbol{0}$ satisfying $-F^*(\bc)=\square$ in a somewhat complicated way, but it is nevertheless totally explicit.\footnote{In \cite[\S13]{H-Bdelta} similar contribution from $\bc\neq\boldsymbol{0}$ also appears in the secondary term in the case where $m=0$, and either $F$ is quaternary and $\Delta_F$ is a perfect square, or $F$ is ternary. However no explanation is provided up to now.} In \cite{K-K} it is interpreted as a certain ``regularised Eisenstein distribution'' (cf. \cite[Remark 1.4]{K-K}). 
It would be interesting to compare the constants obtained via all these different approaches.

 \subsection*{Notation and conventions}  Unless otherwise specified, all implied constants are only allowed depend on $F,L,w$. We introduce the notation
 \begin{equation}\label{eq:omega}
 	\Omega:=2L\Delta_F,\quad m(\Omega):=\prod_{p\mid m,p\nmid \Omega}p.
 \end{equation} By our convention above, $\Omega$ is fixed. For any $l,q\in\BN$, we denote by $q_{l}$ the largest divisor of $q$ whose prime divisors all divide $l$. 
 
 For $n\in\BN$, let $n^\square$ be the largest divisor of $n$ that is square-full. The letter $\chi$ always denotes a multiplicative Dirichlet character, and $|\chi|$ denotes its modulus. We denote by $\chi_*$ the unique primitive character inducing $\chi$. Then $\chi$ is principal if and only if $|\chi_*|=1$.
 The letter $\gamma$ denotes the Euler constant.
Let $\tau(\cdot)$ denote the divisor function. For any $\kappa>0$, we further define the arithmetic function \begin{equation}\label{eq:upsilon}
 	\Upsilon_{\kappa}(n):=\prod_{p\mid n}\left(1-p^{-\kappa}\right)^{-1}.
 \end{equation} 
 
 To highlight the difference between various analogous quantities depending on the non-zero integer $m$ and the homogeneous case (i.e. quadrics defined by $F=0$) in \cite{PART1}, instead of adding everywhere $m$ as a subscript which would make the notation heavier, most of time we put a ``$\wedge$'' symbol above. 
 \section{Recap on Heath-Brown's $\delta$-method}
 
 We start by recalling several basic manipulations from \cite{PART1,PART2} to deal with our inhomogeneous problem. 
 In contrast to \cite{PART1}, in this paper we shall use the letter $d$ to denote the number of variables, which is the dimension of the ambient affine space.

Let $F(\bx)\in\BZ[\bx]$ be a non-degenerate integral quadratic form in $d\geqslant 3$ variables, and let $m$ be a non-zero integer. Let $\CW_m$ be the affine $\BZ$-scheme defined by $(F=m)\subset \BA^d_{\BZ}$. Let $L\in\BN$ and let $\bGamma_m\in\CW_m(\BZ/L\BZ)$. We choose $\blambda_m\in \BZ^d$ a ``lift of $\bGamma_m$'', i.e. such that $F(\blambda_m)\equiv m\bmod L$ and its coordinates have absolute value bounded by $L$.  Fix $w:\BR^d\to\BR$ a compactly supported weight function. We associate the counting function weighted by $w$:
$$N_{\CW_m}(B):=\sum_{\substack{\bx\in\BZ^d,F(\bx)=m\\ \bx\equiv \blambda_m\bmod L}}w\left(\frac{\bx}{B}\right).$$
Write $\bx=L\by+\blambda_m$. Since $$L^2\mid F(L\by+\blambda_m)-m\Leftrightarrow L\mid \widehat{H}(\by),$$ where we define
\begin{equation}\label{eq:hatH}
		\widehat{H}(\by):=\frac{F(\blambda_m)-m}{L}+\nabla F(\blambda_m)\cdot \by,
\end{equation}
 we therefore have
 $$	N_{\CW_m}(B)=\sum_{\substack{\by\in\BZ^d\\ L\mid \widehat{H}(\by)}} w\left(\frac{L\by+\blambda_m}{B}\right)\delta\left(\frac{F(L\by+\blambda_m)-m}{L^2}\right).$$
 
 Applying \cite[Theorem 1]{H-Bdelta} and Poisson summation, we get
\begin{equation}\label{eq:poisson}
	 N_{\CW_m}(B)=\frac{C_Q}{Q^2}\sum_{q=1}^{\infty}\sum_{\bc\in\BZ^d}\frac{\widehat{S}_{q}(\bc)\widehat{I}_{q}(w;\bc)}{(qL)^{d}}, 
\end{equation} where  $C_Q=1+O_N(Q^{-N})$, 	\begin{equation}\label{eq:Sqc}
 	\widehat{S}_{q}(\bc):=\sum_{\substack{a\bmod q\\(a,q)=1}}\sum_{\substack{\bsigma\in(\BZ/qL\BZ)^d\\ \widehat{H}(\boldsymbol{\sigma})\equiv 0\bmod L}}\textup{e}_{qL}\left(a\left(\widehat{H}(\boldsymbol{\sigma})+LF(\boldsymbol{\sigma})\right)+\bc\cdot\bsigma\right)
 \end{equation} is a twisted quadratic exponential sum, and
 	\begin{equation}\label{eq:Iqc}
 	\widehat{I}_{q}(w;\bc):=\int_{\BR^d}w\left(\frac{L\by+\blambda_m}{B}\right)h\left(\frac{q}{Q},\frac{F(L\by+\blambda_m)-m}{L^2Q^2}\right)\e_{qL}\left(-\bc\cdot\by\right)\operatorname{d}\by
 \end{equation} is the ``oscillatory integral''. We emphasize that $\widehat{I}_{q}(w;\bc)$ also depends on $B,Q$, but to keep the notation lighter we omit them. Here $h(x,y):(0,\infty)\times\BR\to\BR$ is a ``nice'' infinitely differentiable function satisfying the properties in \cite[\S4]{H-Bdelta}.

\subsection{Quadratic exponential sums}
Along the same lines as in \cite[\S2.4]{PART1}, 
for $q\in\BN$, let $$q=q_1q_2$$ be a decomposition such that
$$\gcd(q_1,q_2\Omega)=1.$$
Let $$\widehat{k}:=\frac{F(\blambda_m)-m}{L},$$ and write $$\widehat{k}=\widehat{k}_2q_1+\widehat{k}_1q_2L,$$ with $\widehat{k}_1\bmod q_1$ and $\widehat{k}_2\bmod q_2L$. Then using the Chinese remainder theorem, we can decompose $\widehat{S}_{q}(\bc)$ into
$$\widehat{S}_{q}(\bc)=\widehat{S}_{q}^{(1)}(\bc)\widehat{S}_{q}^{(2)}(\bc),$$ where
 \begin{equation}\label{eq:S1}
	\widehat{S}_{q}^{(1)}(\bc):=\sum_{\bsigma_1\bmod q_1}\sum_{\substack{a_1\bmod q_1\\(a_1,q_1)=1}}\e_{q_1}\left(a_1\left((q_2L)^2F(\bsigma_1)+q_2(\nabla F(\blambda_m)\cdot \bsigma_1+\widehat{k}_1)\right)+\bc\cdot\bsigma_1\right),
\end{equation} \begin{equation}\label{eq:S2}
	\widehat{S}_{q}^{(2)}(\bc):=\sum_{\substack{\bsigma_2\bmod q_2L\\ L\mid \widehat{H}(q_1\bsigma_2)}}\sum_{\substack{a_2\bmod q_2\\(a_2,q_2)=1}}\e_{q_2L}\left(a_2\left(q_1^2L F(\bsigma_2)+q_1(\nabla F(\blambda_m)\cdot \bsigma_2+\widehat{k}_2)\right)+\bc\cdot\bsigma_2\right).
\end{equation}  Indeed, $\widehat{S}_{q}^{(1)}(\bc)$ and $\widehat{S}_{q}^{(2)}(\bc)$ are defined in exactly the same way as in \cite[Lemma 4.2]{PART1}, except that we replace all $k_i$ by $\widehat{k}_i$ for $i=1,2$ and in $\widehat{S}_{q}^{(2)}(\bc)$ we replace $H_{\blambda,L}$ by $\widehat{H}$.

The term $\widehat{S}_{q}^{(1)}(\bc)$ \eqref{eq:S1} is an exponential sum over ``good moduli'' and can be estimated explicitly.
\begin{proposition}\label{prop:S1value}
	Assuming $\gcd(q_1,\Omega)=1$, we have
			\begin{multline*}
				\widehat{S}_{q}^{(1)}(\bc)=\e_{q_1}\left(-\overline{q_2L^2}\blambda_m\cdot \bc\right)\\ \times\iota_{q_1}^{d}\left(\frac{\Delta_F}{q_1}\right)\sum_{\substack{a\bmod q_1\\(a,q_1)=1}}\left(\frac{a}{q_1}\right)^{d}\e_{q_1}\left(-am-\overline{4\Delta_F a}(q_2L^2)^2 F^*(\bc)\right).
			\end{multline*}
\end{proposition}
\begin{proof}[Sketch of proof]
	The change of variables
	$$a_1\mapsto a_1':=a_1\overline{L^2},\quad \bsigma_1\mapsto \bb:=q_2L^2\bsigma_1+\blambda_m$$ gives 
	\begin{equation}\label{eq:step1}
		\widehat{S}_{q}^{(1)}(\bc)=\e_{q_1}\left(-\overline{q_2L^2}\bc\cdot\blambda_m\right)T_{q_1}(F,m;\overline{q_2L^2}\bc),
	\end{equation} where we let \begin{equation}\label{eq:T}
T_{q}(F,m;\bc):=\sum_{\substack{a\bmod q\\(a,q)=1}}\sum_{\bb\in(\BZ/q\BZ)^{d}}\e_q\left(a(F(\bb)-m)+\bc\cdot\bb\right).
\end{equation}
	Applying \cite[Lemma 4.3]{PART1} gives that, when $\gcd(q,2\Delta_F)=1$, 
	$$T_{q}(F,m;\bc)=\iota_q^d\left(\frac{\Delta_F}{q}\right)\sum_{\substack{a\bmod q\\(a,q)=1}}\left(\frac{a}{q}\right)^d \e_q\left(-am-\overline{4\Delta_F a}F^*(\bc)\right).$$ 
	We thus obtain the desired formula. 
\end{proof}

As for $\widehat{S}_{q}^{(2)}(\bc)$ \eqref{eq:S2} we have the following uniform bound.
\begin{proposition}\label{prop:S2}
	Uniformly for $q,m$ and $\bc\in\BZ^d$, we have
	$$\widehat{S}_{q}^{(2)}(\bc)\ll q_2^{\frac{d+2}{2}}.$$ 
\end{proposition}
\begin{proof}[Sketch of proof]
	The same proof for $S^{(2)}_{q,L,\blambda}(\bc)$ in \cite[Proposition 4.6]{PART2} (which is uniform in $k$) applies to $\widehat{S}_{q}^{(2)}(\bc)$ (see also the proof of \cite[Lemma 25]{H-Bdelta}). The only difference is that we do not necessarily have $\gcd(\blambda_m,L)=1$. But as long as $L$ is fixed, dropping this assumption only adds a contribution of order $O(1)$, so still gives the estimate above. 
\end{proof}

The trivial bound for average of $\widehat{S}_{q}(\bc)$ over $q\leqslant X$ is $X^{\frac{d+4}{2}}$. We have the following ``Kloosterman refinement'' type result.
\begin{proposition}\label{prop:Kloosterman}
	Let $\psi:\BR_{>0}\to\BR_{>0}$ be a monotone increasing function such that $\psi(x)\to\infty$ as $x\to\infty$. Assume that $d$ is odd and that $-m\Delta_F$ is a square. Then uniformly for $m$ and $\bc\in\BZ^d$ we have, for any $0<\kappa<\frac{1}{2}$, 
	$$\sum_{\substack{q\leqslant X\\q_{m\Omega}>\psi(X)}}\frac{\left|\widehat{S}_{q}(\bc)\right|}{q^{\frac{d+1}{2}}}\ll_{\kappa} \Upsilon_{\kappa}(m) X(\log X)\psi(X)^{-\frac{1}{2}+\kappa}.$$
		If  $-F^*(\bc)$ is not a perfect square, then we have for any $\varepsilon>0$,
		$$\sum_{\substack{q\leqslant X\\ \text{either }q_{\Omega}>\psi(X)\text{ or } q_{m(\Omega)}^\square>\psi(X)}}\frac{\left|\widehat{S}_{q}(\bc)\right|}{q^{\frac{d+1}{2}}}\ll_{\varepsilon} |\bc|^{3+\varepsilon} X\psi(X)^{-\frac{1}{2}+\varepsilon}\min\left(\log X,\Upsilon_{1}(m)\right),$$
		$$\sum_{\substack{q\leqslant X\\q_{m\Omega}>\psi(X)}}\frac{\left|\widehat{S}_{q}(\bc)\right|}{q^{\frac{d+1}{2}}}\ll_{\varepsilon} \Upsilon_{\frac{1}{4}}(m)|\bc|^{2+\varepsilon} X\psi(X)^{-\frac{1}{4}}.$$
\end{proposition}
\begin{remark*}
	The key feature of the case $-F^*(\bc)\neq\square$ is that we dispose the flexibility of either being independent of (the number of prime divisors of) $m$, or saving a factor of $\log X$. See \S3 for asymptotic formulas and estimates for ``sums of Kloosterman sums'' (without absolute values).
\end{remark*}
\begin{proof}
	For $r,s,l\in\BN$ with $\gcd(r,s)=\gcd(rs,\Omega)=1, l\mid (m\Omega)^\infty$ and $\bc\in\BZ^d$ let \begin{equation}\label{eq:CT}
		\CT_r^s(l,\bc):=\sum_{\substack{a\bmod r\\(a,r)=1}}\left(\frac{a}{r}\right)\e_{r}\left(\overline{s}\left(-am-\overline{4\Delta_F a}(lL^2)^2 F^*(\bc)\right)\right).
	\end{equation} We have the Weil bound \begin{equation}\label{eq:Weilbd}
	\CT_r^s(l,\bc)\ll \tau(r)r^{\frac{1}{2}}\gcd\left(r,m,F^*(\bc)\right)^\frac{1}{2}.
	\end{equation} Under the decomposition $q=\frac{q}{q_{m\Omega}}q_{m\Omega}$, the character sum appearing in $\widehat{S}^{(1)}_{q}(\bc)$ equals $\CT_{q/q_{m\Omega}}^1(q_{m\Omega},\bc)$, and hence	using \eqref{eq:Weilbd}, \begin{equation}\label{eq:S1bd}
	\left|\widehat{S}_{q}^{(1)}(\bc)\right| \ll\left(\frac{q}{q_{m\Omega}}\right)^{\frac{d+1}{2}}\tau\left(\frac{q}{q_{m\Omega}}\right)\gcd\left(\frac{q}{q_{m\Omega}},m,F^*(\bc)\right)^\frac{1}{2}=\left(\frac{q}{q_{m\Omega}}\right)^{\frac{d+1}{2}}\tau\left(\frac{q}{q_{m\Omega}}\right).
	\end{equation} 
	Then using \eqref{eq:S1bd} and by Proposition \ref{prop:S2}, for any $0<\kappa<1$, we have (by the well-known bound for average of the divisor function, see e.g. \cite[I.3 Theorem 2]{Tenenbaum})
	\begin{equation}\label{eq:Klsum1}
		\begin{split}
			\sum_{\substack{q\leqslant X\\q_{m\Omega}>\psi(X)}}\frac{\left|\widehat{S}_{q}(\bc)\right|}{q^{\frac{d+1}{2}}}&\ll \sum_{\substack{\psi(X)<q_2\leqslant X\\ q_2\mid(m\Omega)^\infty}}q_2^{\frac{1}{2}} \sum_{\substack{q_1\leqslant \frac{X}{q_2}}}\tau(q_1)\\ &\ll  X\log X\sum_{\substack{\psi(X)<q_2\leqslant X\\ q_2\mid(m\Omega)^\infty}} q_2^{-\frac{1}{2}}\\ &\ll  X(\log X)\Upsilon_{\kappa}(m)\psi(X)^{-\frac{1}{2}+\kappa}.
		\end{split}
	\end{equation}
	
		From now on we assume $-F^*(\bc)\neq\square$. For every $n$, let \begin{equation}\label{eq:rhocdef}
		\rho_{\bc}(n):=\#\{v\bmod n: v^2\equiv -F^*(\bc)\bmod n\}.
	\end{equation} 
	The function $\rho_{\bc}$ is multiplicative by the Chinese remainder theorem. 
	We have, or any $\alpha\geqslant 1$, \begin{equation}\label{eq:Nagell-Ore}
		\rho_{\bc}(p^\alpha)\begin{cases}
			= 1+\left(\frac{-F^*(\bc)}{p}\right)  & \text{ if } p\nmid 2F^*(\bc);\\
			\leqslant 2p^{\frac{1}{2}\operatorname{ord}_p(|4F^*(\bc)|)} & \text{ otherwise},
		\end{cases}
	\end{equation}
	where the equality is shown in \cite[Theorem 50]{Nagell}, and the inequality is due to Huxley \cite{Huxley}\footnote{In fact, a form of the Nagell--Ore theorem is sufficient and gives a larger power on $|\bc|$, see \cite[\S27 Theorems 50--54]{Nagell}.} 
	Hence by multiplicativity we obtain that for any $n\geqslant 1$,
	\begin{equation}\label{eq:rhoc}
		\rho_{\bc}(n)\leqslant 2^{\omega( |2F^*(\bc)|)} |2F^*(\bc)|^\frac{1}{2} \cdot 2^{\omega (n)}\ll_\varepsilon |\bc|^{1+\varepsilon} \cdot 2^{\omega(n)}.
	\end{equation}
	Let $\chi_{\bc}$ denote the Legendre symbol $\left(\frac{-4F^*(\bc)}{\cdot}\right)$. 
	Elementary computations reveal the following relationship between formal Dirichlet series
	$$\sum_{\gcd(n,2F^*(\bc))=1}\frac{\rho_{\bc}(n)}{n^s}=\left(\sum_{\gcd(n,2F^*(\bc))=1}\frac{1}{n^s}\right)\BL(s,\chi_{\bc})\Psi_{\bc}(s),$$ where $\Psi_{\bc}(s):=\prod_{p\nmid 2F^*(\bc)}(1-p^{-2s})$ which is absolutely convergent around $\Re(s)=1$. 
	By the Pólya--Vinogradov inequality coupled with partial summation, we have 
	$$\sum_{n\leqslant X}\frac{\chi_{\bc}(n)}{n}=\sum_{n=1}^{\infty}\frac{\chi_{\bc}(n)}{n}-\sum_{n>X}\frac{\chi_{\bc}(n)}{n}\ll_\varepsilon \left|\BL(1,\chi_{\bc})\right|+\frac{|F^*(\bc)|^{\frac{1}{2}+\varepsilon}}{X^{1-\varepsilon}}\ll_{\varepsilon} |\bc|^{1+\varepsilon}.$$
	Write $$\BL(s,\chi_{\bc})\Psi_{\bc}(s)=\sum_{n=1}^\infty\frac{a_n}{n^s},\quad \Psi_{\bc}(s)=\sum_{n=1}^\infty \frac{b_n}{n^s}.$$ 
	Then for $n$ with $\gcd(n,2F^*(\bc))=1$, $$\rho_{\bc}(n)=\sum_{k\mid n} a_k,\quad a_n=\sum_{k\mid n}\chi_{\bc}(k)b_{\frac{n}{k}}.$$ Now $$\sum_{n\leqslant Y}\frac{a_n}{n}=\sum_{\substack{k,l\in\BN\\ kl\leqslant Y}}\frac{\chi_{\bc}(k)b_l}{kl}\leqslant \sum_{l\leqslant Y}\frac{|b_l|}{l}\left|\sum_{k\leqslant \frac{Y}{l}}\frac{\chi_{\bc}(k)}{k}\right|\ll_{\varepsilon} |\bc|^{1+\varepsilon}.$$ Hence
	\begin{align*}
		\sum_{\substack{n\leqslant Y\\\gcd(n,2F^*(\bc))=1}}\rho_{\bc}(n)&=\sum_{\substack{k\leqslant Y}}a_k\#\left\{{l\leqslant \frac{Y}{k}: \gcd(l,2F^*(\bc))=1}\right\}\\ &=\sum_{\substack{k\leqslant Y}}a_k\left(\left(\sum_{d\mid 2F^*(\bc)}\frac{\mu(d)}{d}\right)\frac{Y}{k}+O_\varepsilon(|\bc|^\varepsilon)\right)\\
		&\ll |\bc|^\varepsilon Y\left(1+\left|\sum_{k\leqslant Y}\frac{a_k}{k}\right|\right)\ll_\varepsilon |\bc|^{1+\varepsilon} Y.
	\end{align*} Finally, using \eqref{eq:rhoc} \begin{equation}\label{eq:rhocsum}
\begin{split}
		\sum_{\substack{n\leqslant Y}}\rho_{\bc}(n)&=\sum_{\substack{n_2\leqslant Y\\ n_2\mid (2F^*(\bc))^\infty}}\rho_{\bc}(n_2)\sum_{\substack{n_1\leqslant \frac{Y}{n_2}\\ \gcd(n_1,2F^*(\bc))=1}}\rho_{\bc}(n_1)\\ &\ll_{\varepsilon} |\bc|^{2+\varepsilon} Y\sum_{\substack{n_2\leqslant Y\\ n_2\mid (2F^*(\bc))^\infty}}\frac{2^{\omega(n_2)}}{n_2}\\ &\ll |\bc|^{2+\varepsilon} Y \times 2^{\omega(2F^*(\bc))}\ll_{\varepsilon} |\bc|^{2+\varepsilon} Y.
	\end{split}	
\end{equation}
	
	
	Now, we decompose $q=\frac{q}{q_{\Omega}}q_{\Omega}$, so that the character sum appearing in $\widehat{S}^{(1)}_{q}(\bc)$ equals $\CT_{q/q_{\Omega}}^1(q_{\Omega},\bc)$. Upon factorising $\frac{q}{q_{\Omega}}=\frac{q}{q_{m\Omega}}\frac{q_{m(\Omega)}}{q_{m(\Omega)}^\square}q_{m(\Omega)}^\square$, by the classical inversion formula $r\overline{r}+s\overline{s}\equiv 1\bmod rs$, we have
	\begin{equation}\label{eq:multiplicativity}
		\CT_{q/q_{\Omega}}^1(q_{\Omega},\bc)=\CT_{q/q_{m\Omega}}^{q_{m(\Omega)}}(q_{\Omega},\bc)\CT_{q_{m(\Omega)}/q_{m(\Omega)}^\square}^{qq_{m(\Omega)}^\square/q_{m\Omega}}(q_{\Omega},\bc)\CT_{q_{m(\Omega)}^\square}^{q/q_{\Omega}q_{m(\Omega)}^\square}(q_{\Omega},\bc). 
	\end{equation}
		We shall analyse $\widehat{S}_{q}^{(1)}(\bc)$ in each of these moduli in more detail.  
		
		Evidently $\gcd(q/q_{m\Omega},m)=1$. On using the explicit evaluation formula for Salié sums \cite[p. 323]{Iwaniec-Kolwalski}, one has, \begin{equation}\label{eq:rSalie}
		\CT_{q/q_{m\Omega}}^{q_{m(\Omega)}}(q_{\Omega},\bc)=\left(\frac{-mq_{m(\Omega)}}{q/q_{m\Omega}}\right)\iota_{q/q_{m\Omega}}\sum_{\substack{u\bmod q/q_{m\Omega}\\ G_{q_{m(\Omega)},q_{\Omega},\bc}(u)\equiv 0\bmod q/q_{m\Omega}}}\e_{q/q_{m\Omega}}(u), 
	\end{equation}
	where for $l_1,l_2\in\BN_{\neq 0}$ and $\bc\in\BZ^d$, we define the quadratic polynomial 
	\begin{equation}\label{eq:G}
		G_{l_1,l_2,\bc}(T):=(\Delta_F l_1 T)^2-m\Delta_F(l_2L^2)^2 F^*(\bc).
	\end{equation}
	Then by the assumption that $-m\Delta_F=\square$, \begin{equation}\label{eq:CTbd1}
		\left|\CT_{q/q_{m\Omega}}^{q_{m(\Omega)}}(q_{\Omega},\bc)\right|\leqslant \left(\frac{q}{q_{m\Omega}}\right)^\frac{1}{2}\rho_{\bc}(q/q_{m\Omega}).
	\end{equation}
	
	Since $q_{m(\Omega)}/q_{m(\Omega)}^\square$ is square-free, then $\CT_{q_{m(\Omega)}/q_{m(\Omega)}^\square}^{qq_{m(\Omega)}^\square/q_{m\Omega}}(q_{\Omega},\bc)$ is a Gauss sum (with respect to the primitive character $\left(\frac{\cdot}{q_{m(\Omega)}/q_{m(\Omega)}^\square}\right)$) and hence $\left|\CT_{q_{m(\Omega)}/q_{m(\Omega)}^\square}^{qq_{m(\Omega)}^\square/q_{m\Omega}}(q_{\Omega},\bc)\right|\leqslant (q_{m(\Omega)}/q_{m(\Omega)}^\square)^\frac{1}{2}$. 
	
	Lastly, by Weil's bound \eqref{eq:Weilbd}, $$\CT_{q_{m(\Omega)}^\square}^{q/q_{\Omega}q_{m(\Omega)}^\square}(q_{\Omega},\bc)\ll_{\varepsilon} (q_{m(\Omega)}^\square)^{\frac{1}{2}+\varepsilon}|\bc|,$$ provided $F^*(\bc)\neq 0$.
	
	Going back to the $q$-sum, by \eqref{eq:rhocsum} and \eqref{eq:multiplicativity}, we obtain that if $F^*(\bc)\neq \square$,
	\begin{equation}\label{eq:Klsum2}
		\begin{split}
			\sum_{\substack{q\leqslant X\\ q_{\Omega}>\psi(X)}}\frac{\left|\widehat{S}_{q}(\bc)\right|}{q^{\frac{d+1}{2}}}\ll_\varepsilon & \sum_{\substack{\psi(X)<q_2\leqslant X\\ q_2\mid \Omega^\infty}}q_2^{\frac{1}{2}}\sum_{\substack{q_1' q_1 ''\leqslant \frac{X}{q_2}\\ q_1'q_1 ''\mid (m(\Omega))^\infty\\ q_1'\square-\text{free},q_1''\square-\text{full}}}|\bc||q_1''|^{\varepsilon}\sum_{\substack{q_1\leqslant\frac{X}{q_1'q_1'' q_2}\\ (q_1,m\Omega)=1}}\rho_{\bc}(q_1)  \\ \ll_\varepsilon &|\bc|^{3+\varepsilon} X\sum_{\substack{\psi(X)<q_2\leqslant X\\ q_2\mid \Omega^\infty}} q_2^{-\frac{1}{2}}\sum_{\substack{q_1' \leqslant X\\q_1'\mid (m(\Omega))^\infty}}q_1^{'-1}\sum_{\substack{q_1''\leqslant X\\q_1''\square-\text{full}}}q_1^{''-1+\varepsilon}\\ \ll_\varepsilon  &|\bc|^{3+\varepsilon} X\psi(X)^{-\frac{1}{2}+\varepsilon}\min\left(\log X,\Upsilon_{1}(m)\right).
	\end{split} 	\end{equation} 
Similarly \begin{align*}
	\sum_{\substack{q\leqslant X\\ q_{m(\Omega)}^\square>\psi(X)}}\frac{\left|\widehat{S}_{q}(\bc)\right|}{q^{\frac{d+1}{2}}}\ll_\varepsilon& |\bc|^{3+\varepsilon} X\sum_{\substack{q_2\leqslant X\\ q_2\mid \Omega^\infty}} q_2^{-\frac{1}{2}}\sum_{\substack{q_1'\leqslant X\\q_1'\mid (m(\Omega))^\infty }}q_1^{'-1} \sum_{\substack{\psi(X)<q_1 ''\leqslant X \\ q_1''\square-\text{full} }}q_1^{''-1+\varepsilon}\\ \ll_\varepsilon  &|\bc|^{3+\varepsilon} X\psi(X)^{-\frac{1}{2}+\varepsilon}\min\left(\log X,\Upsilon_{1}(m)\right).
\end{align*}

To obtain the third estimate, it suffices to modify slightly \eqref{eq:Klsum1} on using \eqref{eq:rhocsum} and on taking $\kappa=\frac{1}{4}$.
	 This finishes the proof.
\end{proof}

We manipulate the $\widehat{S}^{(2)}$-term \eqref{eq:S2}, whose summand may be written as
$$\e_{q_2L^2}\left(a_2(F(q_1L\bsigma_2+\blambda_m)-m)+\bc\cdot L\bsigma_2\right).$$
 The change of variable $\bsigma_2\mapsto \balpha:=q_1L\bsigma_2+\blambda_m$ yields 
\begin{equation}\label{eq:S2CS}
	\widehat{S}_{q}^{(2)}(\bc)=\e_{q_2L^2}\left(-\overline{q_1}\bc\cdot\blambda_m\right) \widehat{\CS}_{q_2}(q_1;\bc),
\end{equation} where, for each fixed $q_2,\bc$ we define for $x\bmod q_2L^2$ with $\gcd(x,q_2L)=1$, 
\begin{equation}\label{eq:Sq2q1bc}
	\widehat{\CS}_{q_2}(x;\bc):=\sum_{\substack{a_2\bmod q_2\\(a_2,q_2)=1}}\sum_{\substack{\balpha\bmod q_2L^2\\F(\balpha)\equiv m\bmod L^2\\ \balpha\equiv \blambda_m\bmod L}}\e_{q_2L^2}\left(a_2 (F(\balpha)-m)+\overline{x}\bc\cdot\balpha\right).
\end{equation}
 For each character $\chi\bmod q_2L^2$, we let \begin{equation}\label{eq:CA}
	\widehat{\CA}_{q_2}(\chi;\bc):=\frac{1}{\phi(q_2L^2)}\sum_{x\bmod q_2L^2}\chi(x)^c \widehat{\CS}_{q_2}(x;\bc).
\end{equation}
Then the orthogonality of characters yields
\begin{equation}\label{eq:hatCS}
	\widehat{\CS}_{q_2}(q_1;\bc)=\sum_{\chi\bmod q_2L^2}\chi(q_1)\widehat{\CA}_{q_2}(\chi;\bc).
\end{equation}
Note also that by Proposition \ref{prop:S2} we trivially have \begin{equation}\label{eq:CAbd}
	\left|\widehat{\CA}_{q_2}(\chi;\bc)\right|\leqslant \max_{x\bmod q_2L^2}\left|\widehat{\CS}_{q_2}(x;\bc)\right|\ll q_2^{\frac{d+2}{2}}.
\end{equation}
\begin{remark*}
	A similar manipulation for the case $m=0$ used in \cite[\S4.1]{PART2}  yields a function of $q_1\bmod L$.
	However when $m\neq 0$ the inhomogeneous form of the equation $F(\balpha)-m$ above somehow ``thickens'' the dependency on the modulus of $q_1$. 
\end{remark*}
We need the following alternative bound for \eqref{eq:CA}, which is better than \eqref{eq:CAbd} whenever $\chi$ is ``close to being primitive'', and thus improves upon the bound \cite[Lemma 25]{H-Bdelta}.
\begin{proposition}\label{prop:CAq2}
	Uniformly for $m,\bc,q_2,\chi$, we have 
	$$\widehat{\CA}_{q_2}(\chi;\bc)\ll_{\varepsilon} \frac{q_2^{\frac{d+2}{2}+\frac{3}{16}+\varepsilon}}{|\chi_*|^{\frac{1}{4}}}$$ where  $|\chi_*|$ is the modulus of the primitive character $\chi_*$ which induces $\chi$.
\end{proposition}
\begin{proof}
	We need to modify the proof of \cite[Proposition 4.6]{PART1} to get the similar sort of power-saving on $q_2$, while simultaneously taking special care of uniformity and character sums involving $\chi$. (Note however $L$ is considered fixed.) We may assume in what follows that $\chi$ is non-principal, otherwise \eqref{eq:CAbd} is sufficient.
	
	Using the Cauchy--Schwarz inequality regarding the $a_2$-sum, we have 
	\begin{equation}\label{eq:CACR}
		\left|\widehat{\CA}_{q_2}(\chi;\bc)\right|\leqslant \frac{1}{\phi(q_2L^2)}\phi(q_2)^\frac{1}{2}\left(\sum_{\substack{a_2\bmod q_2\\ (a_2,q_2)=1}}\CR_{q_2}(a_2;\bc)\right)^\frac{1}{2},
	\end{equation} where  (the meaning of $(*)$ for the $\balpha$-sum is the same as in \eqref{eq:Sq2q1bc})
	\begin{align*}
		\CR_{q_2}(a_2;\bc)&:=\left|\sum_{x\bmod q_2L^2}\chi(x)^c\sum^{(*)}_{\substack{\balpha\bmod q_2L^2}}\e_{q_2L^2}\left(a_2 (F(\balpha)-m)+\overline{x}\bc\cdot\balpha\right)\right|^2\\ &=\left|\sum^{(*)}_{\substack{\balpha\bmod q_2L^2}}\e_{q_2L^2}\left(a_2 (F(\balpha)-m)\right)\Phi_{q_2,\bc}(\balpha;\chi)\right|^2, 
	\end{align*} and we write $$\Phi_{q_2,\bc}(\balpha;\chi):=\sum_{\substack{x\bmod q_2L^2\\ (x,q_2L^2)=1}}\chi(x)\e_{q_2L^2}\left(x\bc\cdot\balpha\right).$$
We are going to show that \begin{equation}\label{eq:claim}
	\CR_{q_2}(a_2;\bc)\ll_\varepsilon \frac{q_2^{d+\frac{19}{8}+\varepsilon}}{|\chi_*|^{\frac{1}{2}}},
\end{equation} which by \eqref{eq:CACR} implies the desired bound for $\widehat{\CA}_{q_2}(\chi;\bc)$.

	We write $\balpha=L\bbeta+\blambda_m$ with $\bbeta\bmod q_2L$. 
	The summation condition $(*)$ on $\balpha$ is equivalent to $$\widehat{H}(\bbeta)\equiv 0\bmod L.$$
	Expanding the square gives
	\begin{align*}
		\CR_{q_2}(a_2;\bc)&=\sum_{\substack{\bbeta_i\bmod q_2L,i=1,2\\\widehat{H}(\bbeta_i)\equiv 0\bmod L}}\e_{q_2L^2}\left(a_2\left((F(L\bbeta_1+\blambda_m)-m)-(F(L\bbeta_2+\blambda_m)-m)\right)\right) \\ &\quad \times\Phi_{q_2,\bc}(L\bbeta_1+\blambda_m;\chi)\Phi_{q_2,\bc}(L\bbeta_2+\blambda_m;\chi)^c\\
		&=\sum_{\substack{\bbeta_i\bmod q_2L,i=1,2\\\widehat{H}(\bbeta_i)\equiv 0\bmod L}}\e_{q_2L^2}\left(a_2\left(\widehat{H}(\bbeta_1)-\widehat{H}(\bbeta_2)+L(F(\bbeta_1)-F(\bbeta_2))\right)\right)
		\\ &\quad\times\sum_{\substack{x_1\bmod q_2L^2\\ (x_1,q_2L^2)=1}}\sum_{\substack{x_2\bmod q_2L^2\\ (x_2,q_2L^2)=1}}\chi(x_1\overline{x_2})\e_{q_2L^2}\left(\bc\cdot(x_1(L\bbeta_1+\blambda_m)-x_2(L\bbeta_2+\blambda_m))\right).
	\end{align*}
	 On introducing the new variables $$\bbeta_0:=\bbeta_1-\bbeta_2,\quad x_0:=x_1-x_2,$$ and on eliminating $\bbeta_2$ and $x_1$, we rearrange the terms so that (the condition for $\bbeta_0$ is $\nabla F(\blambda_m)\cdot\bbeta_0\equiv 0\bmod L$) 
\begin{multline}\label{eq:CR}
	\CR_{q_2}(a_2;\bc)=\frac{1}{L}\sum_{l\bmod L}\e_L(l\widehat{k})\sum_{x_0\bmod q_2L^2}\e_{q_2L^2}\left(x_0\bc\cdot \blambda_m\right) \\ \times\sum_{\substack{\bbeta_0\bmod q_2L\\ \nabla F(\blambda_m)\cdot\bbeta_0\equiv 0\bmod L}}\e_{q_2L}\left(\nabla F(\blambda_m)\cdot \bbeta_0-L F(\bbeta_0)\right)\CJ_{\chi}(x_0;\bbeta_0;\bc)\CH(l;x_0;\bbeta_0;\bc),
\end{multline} where
\begin{align*}
	\CH(l;x_0;\bbeta_0;\bc)&:=\sum_{\bbeta_1\bmod q_2L}\e_{q_2L}\left(\left(q_2l\nabla F(\blambda_m)+L\nabla F(\bbeta_0)+x_0\bc\right)\cdot \bbeta_1\right),
\end{align*} and via the change of variable $x_2\mapsto x_2':=\overline{x_2}$,
$$\CJ_{\chi}(x_0;\bbeta_0;\bc):=\sum_{\substack{x_2'\bmod q_2L^2\\ (x_2',q_2L^2)=1}}\chi(x_0 x_2'+1)\e_{q_2L^2}\left(\overline{x_2'}L\bc\cdot\bbeta_0\right).$$
Here and after, we extend $\chi$ to $\BZ/q_2L^2\BZ$ by letting $\chi(x)=0$ whenever $\gcd(x,q_2L^2)>1$. This exponential sum features the correlation of the variables $x_0,x_2'$. From now on our analysis diverges from \cite{PART1}.

Once $\bc$, $l$ and $x_0$ are fixed, we have $\CH(l;x_0;\bbeta_0;\bc)\neq 0$ only if \begin{equation}\label{eq:divcond}
	q_2L\mid q_2l\nabla F(\blambda_m)+L\nabla F(\bbeta_0)+x_0\bc,
\end{equation} in which case $$\CH(l;x_0;\bbeta_0;\bc)=(q_2L)^{d}.$$
The number of $\bbeta_0\bmod q_2L$ satisfying \eqref{eq:divcond} is $O(1)$ uniformly for all $l$ and $x_0$, because, given that $L$ is fixed, the residue class of the vector $L\nabla F(\bbeta_0)\bmod q_2L$ is uniquely determined.

We now estimate $\CJ_{\chi}(x_0;\bbeta_0;\bc)$ for fixed $\bc$, $x_0$ and $\bbeta_0$. 
 We factorize $$q_2L^2=b_1b_2$$ with $b_1$ square-free, $b_2$ square-full and $\gcd(b_1,b_2)=1$,  and we decompose $$x_0=b_1 x_{0,2}+b_2 x_{0,1}$$ with $x_{0,1}\bmod b_1$ and $x_{0,2}\bmod b_2$. By multiplicativity, we can split uniquely $\chi=\chi_{b_1}\chi_{b_2}$ where $\chi_{b_i}$ is of modulus $b_i,i=1,2$. Then $\CJ_{\chi}(x_0;\bbeta_0;\bc)=\CJ_1 \CJ_2$ where
$$\CJ_1:=\sum_{\substack{y_1\bmod b_1\\\gcd(y_1,b_1)=1}}\chi_{b_1}(1+b_2x_{0,1}y_1)\e_{b_1}\left(\overline{b_2}L\bc\cdot\bbeta_0\overline{y_1}\right),$$
$$\CJ_2:=\sum_{\substack{y_2\bmod b_2\\\gcd (y_2,b_2)=1}}\chi_{b_2}(1+b_1x_{0,2}\overline{y_2})\e_{b_2}\left(\overline{b_1}L\bc\cdot\bbeta_0y_2\right).$$

We next deal with $\CJ_1$. (Whenever $\bc\cdot \bbeta_0\not\equiv  0\bmod b_1$ we may also extend $\e_{b_2}\left((\overline{b_1}L\bc\cdot\bbeta_0)\bullet\right)$ to $\BP^1(\BQ)$ via setting the value of $\infty$ to be $0$.) Write $d_{x_{0,1}}:=\gcd(x_{0,1},b_1)$. We can further factorize $\CJ_1$ into pieces modulo $\frac{b_1}{d_{x_{0,1}}}$ and $d_{x_{0,1}}$ respectively. For each $p\mid \frac{b_1}{d_{x_{0,1}}}$, the polynomial $T\mapsto 1+b_2x_{0,1}T\bmod p$ is non-constant and has a zero over $\BF_p$. So by the Hasse--Weil bound \cite{Weil} (see also \cite[\S11.11]{Iwaniec-Kolwalski})
applied to the sum modulo $\frac{b_1}{d_{x_{0,1}}}$, and  by estimating trivially the sum modulo $d_{x_{0,1}}$, we obtain, uniformly for $\bc$ and $\bbeta_0$,
\begin{equation}\label{eq:CJ1}
	\CJ_1\ll_{\varepsilon} d_{x_{0,1}}\left(\frac{b_1}{d_{x_{0,1}}}\right)^{\frac{1}{2}+\varepsilon}\ll b_1^{\frac{1}{2}+\varepsilon}\gcd(x_0,b_1)^\frac{1}{2}.
\end{equation}

To control $\CJ_2$, we may assume that $$b_2=p^\alpha$$ with $\alpha\geqslant 2$. We apply \cite[Lemmas 12.2 \& 12.3]{Iwaniec-Kolwalski} which yields $$\CJ_2\ll p^{\lceil\frac{\alpha}{2}\rceil}\#\{y\bmod p^{\lfloor\frac{\alpha}{2}\rfloor}:p\nmid b_1 x_{0,2}y+y^2,bb_1 x_{0,2}\overline{(b_1 x_{0,2}y+y^2)}\equiv \overline{b_1} L\bc\cdot\bbeta_0\bmod p^{\lfloor\frac{\alpha}{2}\rfloor}\}\footnote{We take the rational functions of \cite[Lemmas 12.2 \& 12.3]{Iwaniec-Kolwalski} to be $f(y)=\frac{b_1 x_{0,2}}{y}+1,g(y)=\overline{b_1} L\bc\cdot\bbeta_0 y$. Then the function $h$ is $h(y)=\overline{b_1} L\bc\cdot\bbeta_0 -\frac{bb_1 x_{0,2}}{b_1 x_{0,2}y+y^2}$, where $b\bmod p^{\lceil\frac{\alpha}{2}\rceil}$ depends on $\chi$.}$$
The integer $b=b(\chi_{b_2})$ is uniquely determined modulo $p^{\lceil\frac{\alpha}{2}\rceil}$ in \cite[(12.27) \& (12.35)]{Iwaniec-Kolwalski}. More precisely, write $$\beta_1:=\operatorname{ord}_p(b),\quad \alpha_1:=\operatorname{ord}_p(|(\chi_{b_2})_*|).$$ Then $$\beta_1=\begin{cases}
	\lceil\frac{\alpha}{2}\rceil &\text{ if }\alpha_1\leqslant \lfloor\frac{\alpha}{2}\rfloor;\\
	\alpha-\alpha_1 &\text{ if }\alpha_1>\lfloor\frac{\alpha}{2}\rfloor.
\end{cases}$$ In all cases we have $$\beta_1\leqslant \alpha-\alpha_1.$$ 
Let $$\beta_0:=\operatorname{ord}_p(x_{0,2}),\quad \beta:=\beta_0+\beta_1.$$
We first assume $\beta<\lfloor\frac{\alpha}{2}\rfloor$. 
Then we must have $p^{\beta}\mid \overline{b_1} L\bc\cdot\bbeta_0$. On writing $y'\equiv \overline{y}\bmod p^{\lfloor\frac{\alpha}{2}\rfloor-\beta},t_1:=\frac{bb_1 x_{0,2}}{p^\beta},t_2:=\frac{\overline{b_1} L\bc\cdot\bbeta_0}{p^\beta}$, we are led to 
\begin{equation}\label{eq:CJ21}
	\begin{split}
	\CJ_2&\ll p^{\lceil\frac{\alpha}{2}\rceil+\beta}\#\{y'\bmod p^{\lfloor\frac{\alpha}{2}\rfloor-\beta}:y'^2-\overline{t_1}t_2 b_1 x_{0,2}y'-\overline{t_1}t_2\equiv 0 \bmod p^{\lfloor\frac{\alpha}{2}\rfloor-\beta}\}\\ &\ll p^{\lceil\frac{\alpha}{2}\rceil+\frac{1}{2}\lfloor\frac{\alpha}{2}\rfloor+\frac{1}{2}\beta}\leqslant p^{\frac{3}{4}\alpha+\frac{1}{2}\beta+\frac{1}{4}}\leqslant p^{\frac{5}{4}\alpha+\frac{1}{4}+\frac{1}{2}\min(\beta_0,\alpha)-\frac{1}{2}\alpha_1}.
\end{split}
\end{equation} Here we use the estimate \begin{equation}\label{eq:quadcongruence}
\#\{x\bmod p^\gamma:x^2\equiv a\bmod p^\gamma\}\ll p^{\frac{1}{2}\gamma}
\end{equation} for any $a\bmod p^\gamma$. (Indeed, the bound $p^{\frac{1}{2}\gamma}$ is attained precisely when $a=0$, while for $a\neq 0$ it is actually $O(1)$.)
 If $\beta\geqslant \lfloor\frac{\alpha}{2}\rfloor$, then we estimate trivially \begin{equation}\label{eq:CJ22}
 	\CJ_2\ll p^{\alpha}\leqslant p^{\frac{5}{4}\alpha+\frac{1}{4}+\frac{1}{2}\min(\beta_0,\alpha)-\frac{1}{2}\alpha_1}.
 \end{equation}
 
Recall that $b_2$ is square-full. Gathering together \eqref{eq:CJ1} \eqref{eq:CJ21} \eqref{eq:CJ22}, we thus obtain
\begin{align*}
	\CJ_{\chi}(x_0;\bbeta_0;\bc)&\ll_\varepsilon b_1^{\frac{1}{2}+\varepsilon}\gcd(x_0,b_1)^\frac{1}{2}b_2^{\frac{5}{4}+\frac{1}{8}+\varepsilon}\gcd(x_0,b_2)^\frac{1}{2} |(\chi_{b_2})_*|^{-\frac{1}{2}}\\ &\ll b_1^{1+\varepsilon}b_2^{\frac{5}{4}+\frac{1}{8}+\varepsilon}\gcd(x_0,b_1b_2)^\frac{1}{2}|(\chi_{b_1})_*|^{-\frac{1}{2}}|(\chi_{b_2})_*|^{-\frac{1}{2}}\\ &\ll q_2^{\frac{11}{8}+\varepsilon}\gcd(x_0,q_2L^2)^\frac{1}{2}|\chi_*|^{-\frac{1}{2}}.
\end{align*}

Going back to \eqref{eq:CR}, we obtain that  
 \begin{align*}
	\CR_{q_2}(a_2;\bc)\ll_{\varepsilon} \frac{1}{L}\sum_{l\bmod L}\frac{q_2^{\frac{11}{8}+\varepsilon}}{ |\chi_*|^{\frac{1}{2}}}\sum_{\substack{x_0\bmod q_2L^2}}\gcd(x_0,q_2L^2)^\frac{1}{2}\sum_{\substack{\bbeta_0\bmod q_2L\\ \eqref{eq:divcond}\text{ holds}}}q_2^{d}\ll_\varepsilon \frac{q_2^{\frac{11}{8}+1+d+\varepsilon}}{|\chi_*|^{\frac{1}{2}}}.
\end{align*}  The proof of the claim \eqref{eq:claim} is thus finished.
\end{proof}

\subsection{The oscillatory integrals}
We choose \begin{equation}\label{eq:Q}
	Q=\frac{B}{L}.
\end{equation}
	The change of variables $\bt=\frac{L\by+\blambda_m}{B}$ in the integral \eqref{eq:Iqc} gives 
\begin{equation}\label{eq:IJ}
	\widehat{I}_{q}(w;\bc)=\left(\frac{B}{L}\right)^d \e_{qL^2}(\bc\cdot\blambda_m)\widehat{\CI}_{\frac{q}{Q}}^{*}\left(w;\frac{\bc}{L}\right),
\end{equation} where for any $\bb\in\BR^d$ and $r\in\BR_{>0}$ we define
\begin{equation}\label{eq:JqL}
	\widehat{\CI}_{r}^{*}(w;\bb):=\int_{\BR^d}w(\bt)h\left(r,\widehat{F}_B(\bt)\right)\e_r\left(-\bb\cdot\bt\right)\operatorname{d}\bt,
\end{equation}
and we write 
\begin{equation}\label{eq:GB}
	\widehat{F}_B(\bt):=F(\bt)-\frac{m}{B^2}.
\end{equation}
Let us also recall
\begin{equation}\label{eq:CIusual}
	\widehat{\CI}_{r}(w;\bb):=\int_{\BR^d}w(\bt)h\left(r,F(\bt)\right)\e_r\left(-\bb\cdot\bt\right)\operatorname{d}\bt,
\end{equation} which is the ``usual'' one that we have seen in \cite[\S7]{H-Bdelta} (or in a slightly different form in \cite[\S3.1]{PART1}). 

Heath-Brown proves the ``simple estimates'' and ``harder estimates'' \cite[\S7 \& \S8]{H-Bdelta} for $\widehat{\CI}_{r}(w;\bb)$ and its derivative with respect to $r$. (See also  \cite[Corollary 3.3]{PART1}.) As long as $m=O(B^{2-\theta})$ and $B$ is large enough, all the arguments in  \cite[\S7, \S8]{H-Bdelta} regarding $\widehat{\CI}_{r}^{*}(w;\bb)$ work similarly. We summarise these estimates as follows as they will be used repeatedly.
\begin{lemma}\label{le:Isimplehard}
	Assume \eqref{eq:mB2theta}.
	Then uniformly for all $\bb$ and $r$, $$\max\left(\widehat{\CI}_{r}(w;\bb),\widehat{\CI}_{r}^{*}(w;\bb)\right)\ll 1;\quad \text{``Trivial estimate''}.$$ If $\bb=\boldsymbol{0}$ we have
	$$\max\left(\widehat{\CI}_{r}(w;\boldsymbol{0}), r\frac{\partial \widehat{\CI}_{r}(w;\boldsymbol{0})}{\partial r},\widehat{\CI}_{r}^{*}(w;\boldsymbol{0}),r\frac{\partial \widehat{\CI}_{r}^{*}(w;\boldsymbol{0})}{\partial r}\right)\ll 1.$$
	If $\bb\neq \boldsymbol{0}$, we have
$$\max\left(\widehat{\CI}^{*}_{r}(w;\bb),\widehat{\CI}_{r}(w;\bb),r\frac{\partial \widehat{\CI}_{r}(w;\bb)}{\partial r}\right) \begin{cases}
		\ll_{N}  r^{-1}|\bb|^{-N};  & \quad \text{``Simple estimates''},\\
		\ll_{\varepsilon} \left(\frac{r}{|\bb|}\right)^{\frac{d-2}{2}-\varepsilon}; & \quad \text{``Harder estimates''}.
	\end{cases}$$
\end{lemma}

The following lemma measures how close is
$\widehat{\CI}_{r}^{*}(w;\bb)$ to $\widehat{\CI}_{r}(w;\bb)$ as $B\to\infty$.
\begin{lemma}\label{le:CIBCI}
	Assume \eqref{eq:mB2theta}. We have, uniformly for any $0<r\ll 1$ and $\bb\in\BR^d$,
	$$\widehat{\CI}_{r}^{*}(w;\bb)=\widehat{\CI}_{r}(w;\bb)+O(r^{-2}B^{-\theta}).$$
Concerning $\bb=\boldsymbol{0}$ we have
$$\widehat{\CI}_{r}^{*}(w;\boldsymbol{0})=\widehat{\CI}_{r}(w;\boldsymbol{0})+O(B^{-\theta})+O_N(r^{N}).$$
\end{lemma}
\begin{proof}
	We have
	\begin{align*}
		\widehat{\CI}_{r}^{*}(w;\bb)-\widehat{\CI}_{r}(w;\bb)&=\int_{\BR^3}w(\bt)\left(h(r,F(\bt))-h(r,\widehat{F}_B(\bt))\right)\e_r\left(-\bb\cdot\bt\right)\operatorname{d}\bt\\ &\ll \sup_{\bt\in\operatorname{Supp}(w)}\left|\frac{\partial h}{\partial y}(r,F(\bt))\right|\frac{|m|}{B^2}.
	\end{align*}
By \cite[Lemma 5]{H-Bdelta}, the partial derivative above is $\ll r^{-2}$, so the difference above is $\ll r^{-2}B^{-\theta}$. 

We now deal with the case $\bb=\boldsymbol{0}$. 
According to \cite[Lemma 13]{H-Bdelta}, we have
$$\widehat{\CI}_{r}(w;\boldsymbol{0})=\CI(w)+O_N(r^N),$$
$$\widehat{\CI}_{r}^{*}(w;\boldsymbol{0})=\CI^*(w)+O_N(r^N),$$ where $\CI(w)$ is defined by \eqref{eq:CIw} and $\CI^*(w)$ is defined similarly.
Without loss of generality we may assume $\frac{\partial F}{\partial x_1}\gg 1$ in $\operatorname{Supp}(w)$, and we may write $$\CI(w)=\int_{\substack{\bx\in\BR^3\\ x_1\text{ solved by } F(\bx)=0}}w(\bx)\frac{\operatorname{d}\bx}{\frac{\partial F}{\partial x_1}(\bx)},\quad \CI^*(w)=\int_{\substack{\bx\in\BR^3\\ x_1\text{ solved by } F(\bx)=\frac{m}{B^2}}}w(\bx)\frac{\operatorname{d}\bx}{\frac{\partial F}{\partial x_1}(\bx)}.$$
Then
$$\CI(w)-\CI^*(w)\ll \sup_{\bx\in \operatorname{Supp}(w)}\left|\frac{\partial w}{\partial x_1}\left(\frac{\partial F}{\partial x_1}\right)^{-1}(\bx)\right| \frac{|m|}{B}\ll B^{-\theta}.$$
This finishes the proof.
\end{proof}

\section{Equidistribution of certain quadratic congruences}\label{se:saliesum}
In this section we assume $-m\Delta_F=\square$ and let \begin{equation}\label{eq:d0}
	d_0:=\sqrt{-m\Delta_F}.
\end{equation} Recall the polynomial $G_{l_1,l_2,\bc}(T)$ defined in \eqref{eq:G}. Then 
$$G_{l_1,l_2,\bc}(T)=(\Delta_F l_1 T)^2-(d_0l_2L^2)^2(-F^*(\bc)).$$  
So $G_{l_1,l_2,\bc}(T)$ is irreducible if and only if $-F^*(\bc)\neq\square$. 

 Let $\chi$ be a Dirichlet character.  For $X\gg 1$, our main focus is the following average of certain exponential sums against quadratic congruences (recall \eqref{eq:omega}):
\begin{equation}\label{eq:U}
	\CU_{l_1,l_2,\bc}(\chi;X):=\sum_{\substack{n\leqslant X\\ \gcd(n,m\Omega)=1}}\chi(n)\sum_{\substack{v\bmod n\\ G_{l_1,l_2,\bc}(v)\equiv 0\bmod n}}\e_n(v).
\end{equation} 
The crucial technical core of the proof of Theorem \ref{thm:qsum} is the following two results, treated separately depending on whether or not $-F^*(\bc)=\square$.

\begin{theorem}\label{thm:Fstarcneqsq}
	Let $\chi$ be a Dirichlet character. Let $l_1,l_2\in\BN$ be such that $l_i\mid (m\Omega)^\infty,i=1,2$.  
	\begin{enumerate}
		\item 	Let $\bc\in\BZ^3\setminus\boldsymbol{0}$ be such that $-F^*(\bc)\neq\square$. Then uniformly for such $\chi,l_1,l_2,\bc$, we have	$$\CU_{l_1,l_2,\bc}(\chi;X)=O_{\varepsilon}\left(|\bc|^{2+\varepsilon} \frac{X(\log\log X)^{\frac{5}{2}}}{(\log X)^{1-\frac{\sqrt{2}}{2}}}\right).$$
		\item 	Let $\bc\in\BZ^3\setminus\boldsymbol{0}$ be such that  $-F^*(\bc)$ is a non-zero square.  Then there exists $\Gamma(\chi;l_1,l_2;\bc)\in\BC$ such that 
		$$\CU_{l_1,l_2,\bc}(\chi;X)=\Gamma(\chi;l_1,l_2;\bc)X+O_\varepsilon\left((l_1l_2m|\bc|X)^{\varepsilon}(|\chi|^{\frac{11}{8}}l_1^{\frac{7}{8}}X^{\frac{63}{64}}+l_2|\bc|m^\frac{1}{2})\right).$$
		The constant $\Gamma(\chi;l_1,l_2;\bc)$ satisfies $$\Gamma(\chi;l_1,l_2;\bc)\ll_\varepsilon \frac{(l_2|\bc||\chi|)^\varepsilon }{|\chi_*|^{\frac{15}{32}}}\Upsilon_{1}(m)$$ where $\chi_*$ is the primitive character of modulus $|\chi_*|$  inducing $\chi$.
	\end{enumerate}
\end{theorem}

\subsection{Proof of Theorem \ref{thm:Fstarcneqsq} (1)}
Our proof is an adaptation of Hooley's work \cite{Hooley} on equidistribution modulo one of irreducible polynomial congruences. 

Following \cite[p. 40]{Hooley}, for every $n\geqslant 2$ and for every $h\bmod n$,  let us introduce
$$ S_{\bc}(h,n):=\sum_{\substack{v\bmod n\\ v^2\equiv -F^*(\bc)\bmod n}}\e_{n}(hv).$$
By \cite[Lemma 2]{Hooley}, $S_{\bc}(h,n)$ is ``almost'' multiplicative, in the sense that  if $\gcd(n_1,n_2)=1$,
\begin{equation}\label{eq:Smult}
	S_{\bc}(h,n_1n_2)=S_{\bc}(h\overline{n_2},n_1)S_{\bc}(h\overline{n_1},n_2),
\end{equation}
where $\overline{n_1}n_1\equiv 1\bmod n_2$ and $\overline{n_2}n_2\equiv 1\bmod n_1$.

We recall \eqref{eq:rhocdef}.
Clearly for every $h\bmod n$, $$\left|S_{\bc}(h,n)\right|\leqslant \rho_{\bc}(n),$$
and \cite[Lemma 1]{Hooley} provides
\begin{equation}\label{eq:sumS}
	\sum_{h\bmod n}\left|S_{\bc}(h,n)\right|^2\leqslant n\rho_{\bc}(n),
\end{equation} for any $n\geqslant 2$.

We now have $$\CU_{l_1,l_2,\bc}(\chi;X)=\sum_{\substack{n\leqslant X\\ \gcd(n,m\Omega)=1}}\chi(n)S_{\bc}(\overline{\Delta_F l_1} d_0l_2L^2,n).$$
As in \cite[p. 43]{Hooley}, we define $$\CX=\CX(X):=X^{(48\e \log\log X)^{-1}}.$$
And for every $n\in\BN$, we write $$n_{\CX}:=\prod_{\substack{p\mid n\\ p\leqslant \CX}}p^{\operatorname{ord}_p(n)}.$$
Let us split the sum $\CU_{l_1,l_2,\bc}(\chi;X)$ into two parts according to the range of $n_{\CX}$: $$\Sigma_1:=\sum_{\substack{n\leqslant X,n_{\CX}>X^{\frac{1}{3}}\\ \gcd(n,m\Omega)=1}}\chi(n)S_{\bc}(\overline{\Delta_F l_1} d_0l_2L^2,n),~ \Sigma_2:=\sum_{\substack{n\leqslant X,n_{\CX}\leqslant X^{\frac{1}{3}}\\ \gcd(n,m\Omega)=1} } \chi(n)S_{\bc}(\overline{\Delta_F l_1} d_0l_2L^2,n).$$ 

We first consider $\Sigma_1$. Using \eqref{eq:rhoc}, the arguments in \cite[pp. 44--45]{Hooley} (based on the Hardy--Ramanujan inequality) offer \begin{equation}\label{eq:sigma1}
	\left|\Sigma_1\right|\leqslant \sum_{\substack{n\leqslant X\\ n_{\CX}>X^{\frac{1}{3}}}} \rho_{\bc}(n)\ll_\varepsilon |\bc|^{1+\varepsilon}\sum_{\substack{n\leqslant X\\ n_{\CX}>X^{\frac{1}{3}}}} 2^{\omega(n)}=O_\varepsilon\left(|\bc|^{1+\varepsilon}\frac{X}{\log X}\right).
\end{equation}

To deal with $\Sigma_2$, we make use of the multiplicativity \eqref{eq:Smult} to get
\begin{align*}
	\Sigma_2&=\underset{\substack{n\leqslant X,(n,m\Omega)=1\\ n_1=n_{\CX}\leqslant X^{\frac{1}{3}},n_2=n/n_{\CX}}}{\sum\sum}\chi(n)S_{\bc}(\overline{\Delta_F l_1 n_2} d_0l_2L^2,n_1)S_{\bc}(\overline{\Delta_F l_1 n_1}d_0l_2L^2,n_2)\\ &\leqslant \sum_{\substack{n_1\leqslant X^{\frac{1}{3}},\gcd(n_1,m\Omega)=1\\ p\mid n_1\Rightarrow p\leqslant\CX}}T_{n_1,\bc}(X), 
\end{align*} where $$T_{n_1,\bc}(X):=\sum_{\substack{n_2\leqslant \frac{X}{n_1}\\ p\mid n_2\Rightarrow p>\CX}}\rho_{\bc}(n_2)\left| S_{\bc}(\overline{\Delta_F l_1 n_2} d_0l_2L^2,n_1)\right|.$$

We pause to analyse the second moment of the summands in $T_{n_1,\bc}(X)$ for fixed $n_1$. Let $y$ be such that $y\geqslant X^{\frac{2}{3}}$.
The arguments in \cite[p. 46]{Hooley} (using sieve methods \cite[Lemma 7]{Hooley}) plus \eqref{eq:rhoc} offer
$$\sum_{\substack{n_2\leqslant y\\p\mid n_2\Rightarrow p>\CX}} \rho_{\bc}(n_2)^2\ll_\varepsilon |\bc|^{2+\varepsilon} y\frac{(\log\log X)^4}{\log X}.$$
Now for every fixed $n_1$, using \cite[Lemma 8]{Hooley} and \eqref{eq:sumS},
\begin{align*}
	\sum_{\substack{n_2\leqslant y\\ p\mid n_2\Rightarrow p>\CX}}\left| S_{\bc}(\overline{\Delta_F l_1 n_2} d_0l_2L^2,n_1)\right|^2 =&\sum_{\substack{h\bmod n_1\\ (h,n_1)=1}}\left|S_{\bc}(h,n_1)\right|^2\sum_{\substack{n_2\leqslant y,p\mid n_2\Rightarrow p>\CX\\ n_2\equiv \overline{h\Delta_F l_1}d_0l_2L^2\bmod n_1}}1\\ \ll &n_1\rho_{\bc}(n_1)\frac{y}{\phi(n_1)}\frac{\log\log X}{\log X}.
\end{align*}
Therefore, on letting $y=\frac{X}{n_1}$, by Cauchy-Schwarz lemma we have
\begin{align*}
	T_{n_1,\bc}(X) &\ll_\varepsilon |\bc|^{1+\varepsilon}\frac{X(\log\log X)^{\frac{5}{2}}}{\log X}\left(\frac{\rho_{\bc}(n_1)}{n_1\phi(n_1)}\right)^{\frac{1}{2}}.
\end{align*}
Hence  as in \cite[p. 47]{Hooley},
\begin{align*}
	\Sigma_2&\ll  |\bc|^{1+\varepsilon}\frac{X(\log\log X)^{\frac{5}{2}}}{\log X}\sum_{\substack{n_1\leqslant X^{\frac{1}{3}}\\ p\mid n_1\Rightarrow p\leqslant\CX}}\left(\frac{\rho_{\bc}(n_1)}{n_1\phi(n_1)}\right)^{\frac{1}{2}}.\\ &\ll_\varepsilon |\bc|^{1+\varepsilon}\frac{X(\log\log X)^{\frac{5}{2}}}{\log X}\sum_{l\leqslant X}\left(\frac{\rho_{\bc}(l)}{l\phi(l)}\right)^{\frac{1}{2}}\\ &\ll_\varepsilon |\bc|^{1+\varepsilon} 2^{\omega(F^*(\bc))}\frac{X(\log\log X)^{\frac{5}{2}}}{\log X}\prod_{\substack{p\leqslant X\\ p\nmid 2F^*(\bc)}}\left(1+\frac{\rho_{\bc}(p)^{\frac{1}{2}}}{p}\right).
\end{align*}

Our focus in the remaining of the proof is the product over $p$, denoted by $\Pi_{\bc}(X)$ say. Note that for every $p\nmid 2 F^*(\bc)$, by \eqref{eq:Nagell-Ore}, $$\rho_{\bc}(p)=\begin{cases}
	2 &\text{ if }\left(\frac{-F^*(\bc)}{p}\right)=1;\\ 0 &\text{ otherwise}.
\end{cases}$$
Let $K_{\bc}:=\BQ(\sqrt{-F^*(\bc)})$ be the quadratic field extension generated by the roots of $G_{l_1,l_2,\bc}(T)$, and let $$\CP_{\bc}(X):=\{p\leqslant X: p\text{ is unramified and splits in } K_{\bc}\}.$$ In particular $$\Pi_{\bc}(X)=\prod_{\substack{p\in\CP_{\bc}(X)}}\left(1+\frac{\sqrt{2}}{p}\right).$$
Our analysis diverges into the relative size of $|\bc|$ and $\log X$.

\textbf{Case I.} Assume $|\bc|\leqslant \log X$.
Applying the Siegel--Walfisz Theorem (see e.g. \cite[Corollary 5.29]{Iwaniec-Kolwalski}) regarding all quadratic residues modulo $-4F^*(\bc)$, for any $A>2$, we have
\begin{align*}
	\#\CP_{\bc}(X)&=\frac{1}{2}\frac{X}{\log X}+O_A\left(|F^*(\bc)|\frac{X}{(\log X)^A}\right)\\ &=\frac{1}{2}\frac{X}{\log X}+O_A\left(\frac{X}{(\log X)^{A-2}}\right),
\end{align*} by our assumption on the size of $|\bc|$. We may now fix $A=4$.
By partial summation we have
\begin{align*}
	\sum_{\substack{p\in\CP_{\bc}(X)}}\frac{1}{p}&=\sum_{k=2}^{X}\frac{\#\CP_{\bc}(k)-\#\CP_{\bc}(k-1)}{k}+O(1)\\ &=\sum_{k=1}^X\frac{\#\CP_{\bc}(k)}{k^2}+O(1)\\ &=\sum_{k=1}^X\left(\frac{1}{2k\log k}+O\left(\frac{1}{k(\log k)^2}\right)\right)+O(1)\\ &=\frac{1}{2}\log\log X+O(1).
\end{align*}
We therefore conclude, uniformly for such $\bc$,
\begin{align*}
	\Pi_{\bc}(X)&=\exp\left(\sum_{p\in\CP_{\bc}(X)}\log\left(1+\frac{\sqrt{2}}{p}\right)\right)\\ &=O\left(\operatorname{exp}\left(\frac{\sqrt{2}}{2}\log\log X\right)\right)=O\left((\log X)^{\frac{\sqrt{2}}{2}}\right).
\end{align*}

\textbf{Case II.} Assume that $|\bc|>\log X$. We estimate trivially 
$$\Pi_{\bc}(X)\leqslant \prod_{\substack{p\leqslant X}}\left(1+\frac{\sqrt{2}}{p}\right)=O\left((\log X)^{\sqrt{2}}\right)=O\left(|\bc|(\log X)^{\frac{\sqrt{2}}{2}}\right),$$ thanks to our assumption on the size of $|\bc|$. 

So we conclude that in any case $$\Pi_{\bc}(X)=O\left(|\bc|(\log X)^{\frac{\sqrt{2}}{2}}\right),$$
and hence 
\begin{equation}\label{eq:sigma2}
	\Sigma_2\ll_\varepsilon |\bc|^{1+\varepsilon}\frac{X(\log\log X)^{\frac{5}{2}}}{\log X} \Pi_{\bc}(X)\ll_{\varepsilon} |\bc|^{2+\varepsilon}\frac{X(\log\log X)^{\frac{5}{2}}}{(\log X)^{1-\frac{\sqrt{2}}{2}}}. 
\end{equation}
We finally combine the estimates \eqref{eq:sigma1} \eqref{eq:sigma2} to yield the desired bound.
\qed

\subsection{Proof of Theorem \ref{thm:Fstarcneqsq} (2)}
Our proof is inspired by the work of Dartyge--Martin \cite[Theorem 1]{Dartyge-Martin}. In addition to the character sums, the key variant is that we need to control the dependency on $l,m$ and $\bc$ in the error term, 
while dealing with the extra gcd condition on $n$.

We begin with the following slight generalisation of \cite[Lemma 1]{Dartyge-Martin}  concerning mixed exponential sums involving a multiplicative inverse.
\begin{lemma}\label{le:expsum}
	Let $q,s\in \BN$ be such that $\gcd(s,q)=1$ and let $\chi$ be a Dirichlet character such that $|\chi|$ divides $q$. Then there exists $\Theta(q;s,t,\chi)\in\BC$ such that, uniformly for any $t\in\BZ$,
	$$\sum_{\substack{n\leqslant X\\ \gcd(n,sq)=1}}\chi(n)\e_q\left(t\overline{n}\right)=\Theta(q;s,t,\chi)X+O_\varepsilon\left( s^\varepsilon q^{\frac{7}{8}+\varepsilon}\gcd(q,t)^\frac{1}{2}\right).$$
	The constant $\Theta(q;s,t,\chi)$ satisfies $$\Theta(q;s,t,\chi) 	\begin{cases}
	\leqslant \frac{1}{\phi(q/(t,q))} &\text{ if } \chi \text{ is principal}; \\
 \ll_\varepsilon  \frac{q^{\varepsilon}\gcd\left(t,\frac{q}{|\chi_*|}\right)}{q}|\chi_*|^{\frac{1}{2}}&\text{ otherwise}. \end{cases}$$ 
\end{lemma}
\begin{proof}
We proceed via a standard completing process, which is an adaptation of \cite[Lemma 1]{Dartyge-Martin}. We have
\begin{equation}\label{eq:completing}
\begin{split}
	\sum_{\substack{n\leqslant X\\ \gcd(n,sq)=1}}\chi(n)\e_q\left(t\overline{n}\right)&=
	\sum_{\substack{x\bmod q\\ \gcd(x,q)=1}}\chi(x)\e_q(t\overline{x})\sum_{\substack{n\leqslant X,n\equiv x\bmod q\\ \gcd(n,s)=1}}1\\ &=\frac{1}{q}\sum_{a\bmod q}S_{q}(t;a)\sum_{\substack{n\leqslant X\\ \gcd(n,s)=1}}\e_q(an),
\end{split} \end{equation} where for fixed $a\bmod q$, we let $$S_{q}(t;a):=\sum_{\substack{x\bmod q\\ \gcd(x,q)=1}}\chi(x)\e_q(t\overline{x}-ax).$$

We first analyse the main term $a=0$ in \eqref{eq:completing}. 
It is elementary to see that \begin{equation}\label{eq:bds1}
	\#\{n\leqslant X:\gcd(n,s)=1\}=\left(\sum_{d\mid s}\frac{\mu(d)}{d}\right) X+O_\varepsilon(s^\varepsilon).
\end{equation}

We first assume that $\chi$ is principal. Then $S_{q}(t;0)$ is a Ramanujan sum, and can be evaluated as \begin{equation}\label{eq:bds2}
	S_{q}(t;0)=\sum_{\substack{x\bmod q\\(x,q)=1}} \e_{q}\left(tx\right) =\mu\left(\frac{q}{(q,t)}\right)\frac{\phi(q)}{\phi(q/(t,q))}.
\end{equation} 

We now assume that $\chi$ is not principal.
Then $S_{q}(t;0)$ is a Gauss sum, and by \cite[Lemma 3.2]{Iwaniec-Kolwalski},
\begin{align*}
	S_{q}(t;0)&=\sum_{x\bmod q} \chi(x)^c\e_{q}\left(tx\right)\\& =\left(\sum_{x\bmod |\chi_*|}\chi_*(x)^c\e_{|\chi_*|}(x)\right)\left(\sum_{d\mid \gcd\left(t,\frac{q}{|\chi_*|}\right)}d\chi_*\left(\frac{t}{d}\right)\mu\left(\frac{q}{d|\chi_*|}\right)\right).
\end{align*}
It satisfies \begin{equation}\label{eq:bds3}
	S_{q}(t;0)\ll |\chi_*|^{\frac{1}{2}} \sum_{d\mid\gcd\left(t,\frac{q}{|\chi_*|}\right)}d\ll_{\varepsilon} q^{\varepsilon}\gcd\left(t,\frac{q}{|\chi_*|}\right)|\chi_*|^{\frac{1}{2}}.
\end{equation}

 For the terms with $a\neq 0$ in \eqref{eq:completing}, we proceed as the proof of Proposition \ref{prop:CAq2}. We factorize $$q=q_1'q_2'$$ with $q_1'$ square-free, $q_2'$ square-full and $\gcd(q_1',q_2')=1$, so that we decompose $S_{q}(t;a)$ into a mod $q_1'$-sum $\CJ_1'$ and a mod $q_2'$-sum $\CJ_2'$. For $\CJ_1'$, regardless of $\chi$ being principal or not, by the Hasse--Weil estimate \cite{Weil}, we have $$\CJ_1'\ll \tau(q_1')\gcd(a,t,q_1')^\frac{1}{2}q_1^{\prime\frac{1}{2}}.$$  
Now for $\CJ_2'$, we may again assume that $q_2'=p^\alpha$ with $\alpha\geqslant 2$. By \cite[Lemmas 12.2 \& 12.3]{Iwaniec-Kolwalski}, we are led to 
$$\CJ_2'\leqslant p^{\lceil\frac{\alpha}{2}\rceil}\#\{x\bmod p^{\lfloor \frac{1}{2}\alpha \rfloor}:p\nmid x,ax^2-bx+t\equiv 0 \bmod p^{\lfloor \frac{1}{2}\alpha \rfloor}\},$$ where $b=b(\chi)\bmod p^{\lceil\frac{\alpha}{2}\rceil}$ depends on $\chi$. Let $\beta:=\operatorname{ord}_p(\gcd(a,b,t))$. 
If $\beta<\lfloor \frac{1}{2}\alpha \rfloor$, on writing $a':=\frac{a}{p^\beta},b':=\frac{b}{p^\beta},t':=\frac{t}{p^\beta}$, we have
$$\CJ_2'\leqslant p^{\lceil\frac{1}{2}\alpha\rceil+\beta}\#\{x'\bmod p^{\lfloor \frac{1}{2}\alpha \rfloor-\beta}:p\nmid x',a'x^2-b'x+t'\equiv 0 \bmod p^{\lfloor \frac{1}{2}\alpha \rfloor-\beta}\}.$$
If moreover $p\mid a$, then $p\nmid b'^2-4a't'$ (otherwise $p\mid \gcd(a',b',t')$) and the number of such $x'$ is $O(1)$. Hence in this case $\CJ_2'\ll p^{\lceil\frac{1}{2}\alpha\rceil+\beta}$.
Otherwise if $p\nmid a'$, then using \eqref{eq:quadcongruence}, we have $$\CJ_2'\ll p^{\lceil\frac{1}{2}\alpha\rceil+\frac{1}{2}\lfloor\frac{\alpha}{2}\rfloor+\frac{1}{2}\beta}\leqslant p^{\frac{3}{4}\alpha+\frac{1}{2}\beta+\frac{1}{4}}\leqslant p^{\frac{3}{4}\alpha+\frac{1}{4}}\gcd(p^\alpha,a,t)^\frac{1}{2}.$$
If $\beta\geqslant \lfloor \frac{1}{2}\alpha \rfloor$ we estimate trivially $$\CJ_2'\ll p^{\alpha}\leqslant p^{\frac{3}{4}\alpha+\frac{1}{4}}\gcd(p^\alpha,a,t)^\frac{1}{2}.$$
Therefore we conclude that \begin{equation}\label{eq:bdsqta}
	S_{q}(t;a)\ll_{\varepsilon} q^{\frac{7}{8}+\varepsilon}\gcd(q,a,t)^\frac{1}{2}\leqslant q^{\frac{7}{8}+\varepsilon}\gcd(q,t)^\frac{1}{2}.
\end{equation}

Using \eqref{eq:bdsqta}, the contribution from the terms $a\neq 0$ is,
\begin{align*}
	\ll_\varepsilon & q^{\frac{7}{8}+\varepsilon}\gcd(q,t)^\frac{1}{2} \sum_{\substack{a\bmod q\\a\neq 0}}\frac{1}{q}\left|\sum_{\substack{n\leqslant X\\ \gcd(n,s)=1}}\e_q(an)\right|\\ \ll & q^{\frac{7}{8}+\varepsilon}\gcd(q,t)^\frac{1}{2} \sum_{d\mid s}\mu^2(d)\times \frac{1}{q}\sum_{\substack{a\bmod q\\ a\neq 0}}\left|\sum_{n'\leqslant X/d}\e_q(adn')\right|.
\end{align*} We recall that $\gcd(s,q)=1$, and hence $\gcd(d,q)=1$, so the change of variable $a\mapsto ad$ yields that the inner $a$-sum and $n'$-sum is, uniformly for $d\mid s$, $$\sum_{\substack{b\bmod q\\ b\neq 0}} \left|\sum_{n'\leqslant X/d}\e_q(bn')\right|\ll 
\sum_{c=1}^{\frac{q}{2}}\frac{q}{c}\ll q\log q,$$ cf. the argument in \cite[p. 5]{Dartyge-Martin}. 
So the overall contribution is
$$\ll_\varepsilon  s^\varepsilon q^{\frac{7}{8}+\varepsilon}\gcd(q,t)^\frac{1}{2}.$$

We now finish the proof by taking $$\Theta(q;s,t,\chi):=\frac{S_{q}(t;0)}{q}\sum_{d\mid s}\frac{\mu(d)}{d}, $$  and the bound results from \eqref{eq:bds2} and \eqref{eq:bds3}. (Note that $\sum_{d\mid s}\frac{\mu(d)}{d}=\prod_{p\mid s}\left(1-\frac{1}{p}\right)\leqslant 1$.)
\end{proof}

	We now enter the proof of Theorem \ref{thm:Fstarcneqsq} (2) by setting some notational preliminaries.
	We write $$ -F^*(\bc)=\CN(\bc)^2$$ for $\CN(\bc)\in\BN$, and 
	\begin{equation}\label{a0b0}
		a_0:=\frac{l_1 \Delta_F}{\gcd(\Delta_Fl_1,d_0l_2L^2\CN(\bc))},\quad b_0:=\frac{d_0l_2L^2\CN(\bc)}{\gcd(\Delta_Fl_1,d_0l_2L^2\CN(\bc))}.
	\end{equation}
	We shall consider the auxiliary polynomial $$\widehat{G}_{l_1,l_2,\bc}(T):=(a_0 T+b_0)(a_0 T-b_0).$$ The discriminant is $$\CD:=2|a_0b_0|\asymp l_1l_2|\bc| |m|^\frac{1}{2}.$$ Note that $$G_{l_1,l_2,\bc}(T)=\gcd(\Delta_F l_1,d_0l_2L^2\CN(\bc))^2\widehat{G}_{l_1,l_2,\bc}(T)$$ and $\gcd(a_0,b_0)=1$.
	Thanks to the condition $\gcd(n,m\Omega)=1$, we have
	$$G_{l_1,l_2,\bc}(u)\equiv 0\bmod n\Leftrightarrow \widehat{G}_{l_1,l_2,\bc}(u)\equiv 0\bmod n.$$
	(So that the coefficients have gcd one.)
	In view of \cite[Lemmas 2\&3]{Dartyge-Martin}, upon prescribing $\gcd(n,\CD)$ then separating the congruence conditions of the two factors of $\widehat{G}_{l_1,l_2\bc}$, we have 
	\begin{align*}
		\CU_{l_1,l_2,\bc}(\chi;X)&=\sum_{\substack{g\mid \CD\\ (g,a_0)=1}}\mu^2(g)\sum_{\substack{n\leqslant \frac{X}{g}\\ (n,\CD m \Omega)=1}}\chi(n)\sum_{\substack{r\bmod ng\\ \widehat{G}_{l,\bc}(r)\equiv 0\bmod ng}}\e_{ng}\left(r\right)\\ &=\sum_{\substack{g\mid \CD\\ (g,a_0)=1}}\mu^2(g)\sum_{\substack{n\leqslant \frac{X}{g}\\ (n,\CD m \Omega)=1}}\chi(n)\sum_{\substack{j,k:n=jk\\ (j,k)=(j,a_0)=(k,a_0)=1}}\e_{ng}\left(r_{j,k,g}\right),
	\end{align*} where $r_{j,k,g}\bmod ng$ is uniquely determined by the congruence conditions
	$$a_0r_{j,k,g}+b_0\equiv 0\bmod k, \quad a_0r_{j,k,g}-b_0\equiv 0\bmod j,\quad r_{j,k,g}\equiv b_0\overline{a_0}\bmod g.$$
	
	If $X\leqslant b_0$, then we sum everything trivially and obtain
	$$\CU_{l_1,l_2,\bc}(\chi;X)\ll_\varepsilon \sum_{\substack{g\mid \CD\\ (g,a_0)=1}}\mu^2(g) \sum_{n\leqslant b_0}n^\varepsilon \ll b_0^{1+\varepsilon}.$$
	We next assume $\frac{X}{g}> b_0$ for certain $g\mid \CD$ and we separate the sum over $j$ and $k$ respectively. Let $\mathfrak{S}_1(\chi;X)$ be the sum with the restriction $k\leqslant\sqrt{ \frac{X}{g} }$ and let $\mathfrak{S}_2(\chi;X)$ denote the sum with the remaining range. Proceeding as in \cite[p. 7, the equation above (7)]{Dartyge-Martin} we obtain
	$$\mathfrak{S}_1(\chi;X)=\sum_{\substack{g\mid \CD\\ (g,a_0)=1}}\mu^2(g)\sum_{\substack{k\leqslant\sqrt{\frac{X}{g}}\\ (k,\CD m \Omega)=1}}\chi(k)\sum_{\substack{j\leqslant \frac{X}{gk}\\ (j,k\CD m \Omega)=1}}\chi(j)\e_{a_0k}\left(\widehat{t}\cdot\overline{j}\right)\e_{a_0gjk}(b_0).$$
	Here 
	$\widehat{t}=\widehat{t}(g,k,a_0,b_0;\bc)\bmod a_0k$ which is independent of $j$ and satisfies  $\gcd(\widehat{t},a_0k)=1$. 
	Thanks to $\e_{a_0gjk}(b_0)=1+O\left(\frac{b_0}{a_0gjk}\right)$,  we obtain further
\begin{equation}\label{eq:S1chara}
	\mathfrak{S}_1(\chi;X)=\sum_{\substack{g\mid \CD\\ (g,a_0)=1}}\mu^2(g)\sum_{\substack{k\leqslant\sqrt{\frac{X}{g}}\\ (k,\CD m \Omega)=1}}\chi(k)\sum_{\substack{j\leqslant \frac{X}{gk}\\ (j,k\CD m \Omega)=1}}\chi(j)\e_{a_0k}\left(\widehat{t}\cdot\overline{j}\right)+O_\varepsilon(b_0^{1+\varepsilon}(\log X)^2).
\end{equation} 

 We set $$q:=\lcm(|\chi|,a_0k),\quad s:=\prod_{p\mid m\CD\Omega,p\nmid q}p,\quad t:=\frac{q}{a_0k}\widehat{t}.$$ Observe that $$\gcd(q,t)=\frac{q}{a_0k}\leqslant |\chi|.$$ On applying Lemma \ref{le:expsum}, there exists $\Theta(k;\chi)$ depending also on $a_0,b_0$ such that the inner $j$-sum in \eqref{eq:S1chara} equals
$$\frac{X}{gk}\Theta(k;\chi)+O_\varepsilon\left((m\CD)^\varepsilon|\chi|^{\frac{11}{8}+\varepsilon}(a_0k)^{\frac{7}{8}+\varepsilon}\right),$$ whence $$\mathfrak{S}_1(\chi;X)=X\sum_{\substack{g\mid \CD\\ (g,a_0)=1}}\frac{\mu^2(g)}{g}\sum_{\substack{k\leqslant\sqrt{\frac{X}{g}}\\ (k,\CD m \Omega)=1}}\lambda_1(k)+O_\varepsilon\left((m\CD)^\varepsilon|\chi|^{\frac{11}{8}+\varepsilon}a_0^{\frac{7}{8}+\varepsilon}  X^{\frac{15}{16}+\varepsilon}+b_0^{1+\varepsilon}X^\varepsilon\right).$$
where $$\lambda_1(k):=\frac{\chi(k)\Theta(k;\chi)}{k}.$$
Moreover, on observing $$\gcd\left(t,\frac{q}{|\chi_*|}\right)\leqslant\min\left(\frac{q}{a_0k},\frac{q}{|\chi_*|}\right)\leqslant \frac{q}{|\chi_*|^{\frac{31}{32}}(a_0k)^\frac{1}{32}},$$
Lemma \ref{le:expsum} also gives, \begin{equation}\label{eq:thetabd}
	\Theta(k;\chi)\begin{cases}
		\leqslant \frac{1} {\phi(k)}\ll_{\varepsilon}\frac{1}{k^{1-\varepsilon}} & \text{ if } \chi \text{ is principal}; \\  \ll_{\varepsilon}\frac{|\chi|^\varepsilon }{|\chi_*|^{\frac{15}{32}}(a_0k)^{\frac{1}{32}-\varepsilon}} &\text{ if } \chi \text{ is not principal}.
	\end{cases}
\end{equation}  Therefore
$$\sum_{\substack{k\leqslant\sqrt{\frac{X}{g}}\\ (k,\CD m\Omega)=1}}\lambda_1(k)=\sum_{\substack{k\in\BN\\ (k,\CD m\Omega)=1}}\lambda_1(k)+O_\varepsilon(|\chi|^\varepsilon g^\frac{1}{64}X^{-\frac{1}{64}+\varepsilon}).$$
So extending the $k$-sum to infinity results in $O_\varepsilon((\CD|\chi|)^\varepsilon X^{\frac{63}{64}+\varepsilon})$. So, on inserting the error term in \eqref{eq:S1chara}, we henceforth conclude that \begin{multline}\label{eq:sum1}
			\mathfrak{S}_1(\chi;X)=X\left(\sum_{\substack{g\mid \CD\\ (g,a_0)=1}}\frac{\mu^2(g)}{g}\right)\left(\sum_{\substack{k\in\BN\\ (k,\CD m \Omega)=1}}\lambda_1(k)\right)\\ +O_\varepsilon\left((m\CD)^\varepsilon|\chi|^{\frac{11}{8}+\varepsilon}a_0^{\frac{7}{8}+\varepsilon}X^{\frac{63}{64}+\varepsilon}+b_0^{1+\varepsilon}(\log X)^2\right).
\end{multline}

As for $\mathfrak{S}_2(\chi;X)$, we just need to reverse the role of $j$ and $k$, and we deduce that $\mathfrak{S}_2(\chi;X)$ has the same asymptotic formula as \eqref{eq:sum1}.

Therefore we conclude the proof of Theorem \ref{thm:Fstarcneqsq} (2) on letting \begin{equation}\label{eq:Gamma}
	\Gamma(\chi;l_1,l_2;\bc):=2\sum_{\substack{g\mid \CD\\ (g,a_0)=1}}\frac{\mu^2(g)}{g}\sum_{\substack{k\in\BN\\(k,\CD m\Omega)=1}}\lambda_1(k).
\end{equation}
Then by \eqref{eq:thetabd}, $$\Gamma(\chi;l_1,l_2;\bc)\ll_{\varepsilon}\frac{|\chi|^\varepsilon}{|\chi_*|^{\frac{15}{32}}}\Upsilon_{1}(2b_0)\ll_{\varepsilon} \frac{(l_2|\bc||\chi|)^\varepsilon }{|\chi_*|^{\frac{15}{32}}}\Upsilon_{1}(m)$$ as desired.
\qed

\section{Sums of Salié sums}

We specialise from now on to the case where $d=3$ and $-m\Delta_F=\square$.

Recall \eqref{eq:Sqc}. For $\bc\in\BZ^3$ and $X\gg 1$, let us define \begin{equation}\label{eq:CFc}
	\CF_{\bc}(X):=\sum_{\substack{q\leqslant X}}\frac{\e_{qL^2}(\bc\cdot\blambda_m)\widehat{S}_{q}(\bc)}{q^2}.
\end{equation}

Based on results established in \S\ref{se:saliesum}, our goal is prove the following evaluation of ``sums of Salié sums''.

\begin{theorem}\label{thm:sumsaliesum}
	Let $\bc\in\BZ^3\setminus\boldsymbol{0}$. 
	\begin{enumerate}
		\item Assume that $-F^*(\bc)$ is not a perfect square. Then
		$$\CF_{\bc}(X)\ll_\varepsilon|\bc|^{3+\varepsilon}X \min\left(\Upsilon_{\frac{1}{4}}(m)(\log X)^{-\frac{1}{4}(1-\frac{\sqrt{2}}{2})+\varepsilon},(\log X)^{\frac{1}{4}(3+\frac{\sqrt{2}}{2})+\varepsilon}\right).$$
		\item	Assume that $-F^*(\bc)$ is a perfect square (including $0$). Then 
		$$\CF_{\bc}(X)\begin{cases}
			&=\eta(\bc) X+O_{\varepsilon,\beta}\left((mX)^\varepsilon |\bc|^{1+\varepsilon} (X^{\frac{43}{8}\beta }(X^{\frac{63}{64}}+m^\frac{1}{2})+X^{1-\frac{1}{32}\beta})\right);\\
			&\ll_{\varepsilon} |\bc|^{1+\varepsilon} X\left(\Upsilon_{\frac{31}{64}}(m)(\log X)\exp\left(-\frac{1}{64}(\log X)^{\frac{33}{64}}\right)+(\log X)^{\frac{1983}{1984}}\right).
		\end{cases}$$ 
		for any fixed $0<\kappa_1<1$ and $\beta>0$, where $\eta(\bc)\in\BR$ is defined in \eqref{eq:etac2} \eqref{eq:etac1} and satisfies $\eta(\bc)\ll \Upsilon_{\kappa_0}(m)$ for certain numerical constant $0<\kappa_0<1$.
	\end{enumerate}
\end{theorem}
\begin{remark*}
	The exponents of these bounds, though being sufficient for the goal of the current article, are by no means optimal.
\end{remark*}
\noindent\emph{Proof of Theorem \ref{thm:sumsaliesum}.} We shall exhibit two alternative bounds relying on different decompositions for the $q$-sum $\CF_{\bc}(X)$. We shall begin with the proof for the easier bound, before entering into the argument for the harder one.
	\subsection{} 
	By Proposition \ref{prop:S1value} and by \eqref{eq:rSalie}, on recalling \eqref{eq:S2CS}--\eqref{eq:hatCS}, \eqref{eq:U}, and on using the formula
	$$\overline{q_1}q_1+\overline{q_2L^2}q_2L^2\equiv 1\bmod q_1q_2L^2$$ whenever $\gcd(q_1,q_2L)=1$,
	we have the following decomposition
	\begin{align}
		\CF_{\bc}(X)&=\sum_{\substack{q_2\leqslant X\\q_2\mid(m\Omega)^\infty}}\frac{1}{q_2^2}\sum_{\chi\bmod q_2L^2}\widehat{\CA}_{q_2}(\chi;\bc)\sum_{\substack{q_1\leqslant \frac{X}{q_2}\\ \gcd(q_1,m\Omega)=1}}\chi(q_1)\sum_{\substack{u\bmod q_1\\ G_{1,q_2,\bc}(u)\equiv 0\bmod q_1}}\e_{q_1}(u)\nonumber\\ &=\sum_{\substack{q_2\leqslant X\\q_2\mid(m\Omega)^\infty}}\frac{1}{q_2^2}\sum_{\chi\bmod q_2L^2}\widehat{\CA}_{q_2}(\chi;\bc)\CU_{1,q_2,\bc}\left(\chi;\frac{X}{q_2}\right).\label{eq:decomp1}
	\end{align}
	
	\subsubsection{Case where $-F^*(\bc)\neq\square$.} Let 	$\CF_{\bc}^{(1)}(\beta;X)$ denote the sum in \eqref{eq:CFc} restricted to the range $q_{m\Omega}\leqslant (\log X)^\beta$.
According to Proposition \ref{prop:Kloosterman} applied to $\psi(X):=(\log X)^\beta$, we have
			$$\CF_{\bc}(X)-\CF_{\bc}^{(1)}(\beta;X)\ll_{\varepsilon} \Upsilon_{\frac{1}{4}}(m) |\bc|^{2+\varepsilon}X(\log X)^{-\frac{1}{4}\beta+\varepsilon}.$$
			
				We next focus on $\CF_{\bc}^{(1)}(\beta;X)$. 
			Applying Theorem \ref{thm:Fstarcneqsq} (1) we obtain $$\CU_{1,q_2,\bc}\left(\chi,\frac{X}{q_2}\right)\ll_\varepsilon |\bc|^{2+\varepsilon} \frac{X}{q_2(\log X)^{1-\frac{\sqrt{2}}{2}-\varepsilon}}.$$ Hence using the decomposition \eqref{eq:decomp1} and by Proposition \ref{prop:S2},
			\begin{align*}
				\CF_{\bc}^{(1)}(\beta;X)&\ll_\varepsilon  |\bc|^{2+\varepsilon} \frac{X}{(\log X)^{1-\frac{\sqrt{2}}{2}-\varepsilon}}\sum_{\substack{q_2\leqslant (\log X)^\beta\\ q_2\mid (m\Omega)^\infty}}\frac{\sup_{\chi\bmod q_2L^2}\left|\widehat{\CA}_{q_2}(\chi;\bc)\right|}{q_2^2}\\  &\ll_\varepsilon|\bc|^{2+\varepsilon} \frac{X}{(\log X)^{1-\frac{\sqrt{2}}{2}-\varepsilon}}\sum_{\substack{q_2\leqslant (\log X)^\beta\\ q_2\mid(m\Omega)^\infty}}q_2^{\frac{1}{2}}\ll \Upsilon_{\frac{1}{4}}(m) |\bc|^{2+\varepsilon} \frac{X}{(\log X)^{1-\frac{\sqrt{2}}{2}-\frac{3}{4}\beta-\varepsilon}}.
			\end{align*}
			It suffices to choose $\beta=1-\frac{\sqrt{2}}{2}$ to conclude this case.

	\subsubsection{Case where $-F^*(\bc)$ is a non-zero square.}
	Let us denote by $\CF_{\bc}^{(2)}(\beta;X)$ the sum in \eqref{eq:CFc} restricted to $q_{m\Omega}\leqslant X^\beta$.
	Then Proposition \ref{prop:Kloosterman} shows that 
	\begin{equation}\label{eq:CFbd1}
		\CF_{\bc}(X)-\CF_{\bc}^{(2)}(\beta;X)\ll_{\varepsilon}m^\varepsilon |\bc|^{1+\varepsilon} X^{1-\frac{1}{4}\beta+\varepsilon}.
	\end{equation}
	By Theorem \ref{thm:Fstarcneqsq} (2), we obtain that
	$$\CU_{1,q_2,\bc}\left(\chi;\frac{X}{q_2}\right)=\frac{\Gamma(\chi;1,q_2;\bc)}{q_2}X+O_\varepsilon\left((mX)^\varepsilon |\bc|^{1+\varepsilon}q_2^{\frac{11}{8}}(X^{\frac{63}{64}}+m^\frac{1}{2})\right).$$ The big $O$-term contributes to $\CF_{\bc}^{(2)}(\beta;X)$, by \eqref{eq:CAbd}
	\begin{align*}
		&\ll_{\varepsilon} (m X)^\varepsilon|\bc|^{1+\varepsilon}(X^{\frac{63}{64}}+m^\frac{1}{2})\sum_{\substack{q_2\leqslant X^\beta\\ q_2\mid(m\Omega)^\infty}}q_2^\frac{3}{8}\sup_{\chi\bmod q_2L^2}\left|\widehat{\CA}_{q_2}(\chi;\bc)\right| \\ &\ll_{\varepsilon} m^\varepsilon |\bc|^{1+\varepsilon} X^{\frac{43}{8}\beta+\varepsilon }(X^{\frac{63}{64}}+m^\frac{1}{2}).
	\end{align*}
	By Theorem \ref{thm:Fstarcneqsq} (2) 
	we have \begin{equation}\label{eq:bdgamma}
		\Gamma(\chi;1,q_2;\bc)\ll_{\varepsilon} \frac{(m|\bc|q_2)^\varepsilon }{|\chi_*|^\frac{15}{32}}.
	\end{equation}
Combing \eqref{eq:bdgamma} with Proposition \ref{prop:CAq2}, we obtain
	\begin{equation}\label{eq:chisum}
			\begin{split}
			&\sum_{\substack{q_2>X^\beta\\ q_2\mid (m\Omega)^\infty}}\sum_{\chi\bmod q_2L^2}\frac{\widehat{\CA}_{q_2}(\chi;\bc)\Gamma(\chi;1,q_2;\bc)}{q_2^3}\\ \leqslant&\sum_{\substack{q_2>X^\beta\\ q_2\mid (m\Omega)^\infty}}\sum_{e\mid q_2L^2}\sum_{\substack{\chi\bmod q_2L^2\\ |\chi_*|=e}}\frac{\left|\widehat{\CA}_{q_2}(\chi;\bc)\Gamma(\chi;1,q_2;\bc)\right|}{q_2^3} \\ \ll_{\varepsilon}&(m|\bc|)^\varepsilon \sum_{\substack{q_2>X^\beta\\ q_2\mid (m\Omega)^\infty}}q_2^{-\frac{5}{16}+\varepsilon}\sum_{e\mid q_2L^2}e^\frac{9}{32}\\ \ll_{\varepsilon} & (m|\bc|)^\varepsilon \sum_{\substack{q_2>X^\beta\\ q_2\mid (m\Omega)^\infty}} q_2^{-\frac{1}{32}+\varepsilon}\ll_{\varepsilon}  (m|\bc|)^\varepsilon X^{-\frac{1}{32}\beta+\varepsilon}.
		\end{split}
	\end{equation}
	So the leading term in $\CU_{1,q_2,\bc}\left(\chi;\frac{X}{q_2}\right)$ contributes
	$$X\sum_{\substack{q_2\leqslant X^\beta\\ q_2\mid(m\Omega)^\infty}}\sum_{\chi\bmod q_2L^2}\frac{\widehat{\CA}_{q_2}(\chi;\bc)\Gamma(\chi;1,q_2;\bc)}{q_2^3}=\eta(\bc)X+O_\varepsilon\left((m|\bc|)^\varepsilon X^{1-\frac{1}{32}\beta+\varepsilon}\right),$$ where for $\bc\in\BZ^3\setminus\boldsymbol{0}$ with $-F^*(\bc)=\square\neq 0$, we let 
	\begin{equation}\label{eq:etac2}
		\eta(\bc):=\sum_{\substack{u\mid (m\Omega)^{\infty}}}\sum_{\chi\bmod uL^2}\frac{\widehat{\CA}_{u}(\chi;\bc)\Gamma(\chi;1,u;\bc)}{u^3}.
	\end{equation}
It satisfies $$\eta(\bc)\ll_{\varepsilon} \sum_{u\mid (m\Omega)^\infty}u^{-\frac{1}{32}+\varepsilon}\ll \Upsilon_{\frac{1}{32}-\varepsilon}(m)$$ by the argument above. Hence for such $\bc$,  $$\CF_{\bc}^{(2)}(\beta;X)=\eta(\bc)X+O_\varepsilon\left((mX)^\varepsilon |\bc|^{1+\varepsilon} (X^{\frac{43}{8}\beta }(X^{\frac{63}{64}}+m^\frac{1}{2})+X^{1-\frac{1}{32}\beta})\right).$$

	\subsubsection{Case where $F^*(\bc)=0$.}
		We pause to evaluate \eqref{eq:CT} assuming $\gcd(r,m\Omega)=1$ and $F^*(\bc)=0$. In view of \eqref{eq:rSalie}, it suffices to evaluate the function \begin{equation}\label{eq:CV}
		\CV(q):=\sum_{\substack{u^2\equiv 0\bmod q}} \e_q(u).
	\end{equation} We first show that it is multiplicative. Indeed, if $q=q_1q_2$ with $\gcd(q_1,q_2)=1$, then on factorising $u=q_2u_1+q_1u_2$ with $u_1\bmod q_1,u_2\bmod q_2$, we have, by the Chinese remainder theorem,
	\begin{align*}
		\CV(q)&=\sum_{\substack{u_1 \bmod q_1,u_2\bmod q_2\\ (q_2u_1+q_1u_2)^2\equiv 0\bmod q_1q_2}}\e_{q_1}(u_1)\e_{q_2}(u_2)\\ &=\CV(q_1)\CV(q_2).
	\end{align*}
	For $\alpha\in\BN$, a straightforward computation gives
	$$\CV(p^\alpha)=\sum_{\substack{u'\bmod p^{\alpha-\lceil\frac{\alpha}{2}\rceil}}}\e_{p^{\alpha-\lceil\frac{\alpha}{2}\rceil}}(u')=\begin{cases}
		1 &\text{ if } \alpha=1;\\ 0 &\text{ otherwise}.
	\end{cases}$$ We conclude \begin{equation}\label{eq:CVvalue}
		\CV(q)=\mu(q)^2.
	\end{equation} 
	Therefore, in view of the computation \eqref{eq:CV} \eqref{eq:CVvalue}, we have
	$$\CU_{1,q_2,\bc}\left(\chi;Y\right)=\sum_{\substack{q_1\leqslant Y\\ \gcd(q_1,m\Omega)=1}}\mu^2(q_1)\chi(q_1).$$
	
	 On using the elementary identity $\mu^2(n)=\sum_{e^2\mid n}\mu(e)$ and on using the Burgess's $\frac{3}{16}$-bound, we obtain, for each fixed $q_2\mid(m\Omega)^\infty$ and non-principal $\chi\bmod q_2L^2$,
	\begin{equation}\label{eq:fstarc=01}
			\begin{split}
			\CU_{1,q_2,\bc}\left(\chi;Y\right)&=\sum_{d\mid m\Omega}\mu(d)\sum_{q_1\leqslant Y,d\mid q}\mu^2(q_1)\chi(q_1)\\ &=\sum_{d\mid m\Omega}\mu(d)\chi(d)\sum_{e\leqslant (\frac{Y}{d})^\frac{1}{2}}\mu(e)\chi(e)^2\sum_{n\leqslant \frac{Y}{d e^2}}\chi(n)\ll_{\varepsilon} m^\varepsilon Y^{\frac{1}{2}+\varepsilon} q_2^{\frac{3}{16}+\varepsilon}. 
		\end{split}
	\end{equation}
	Now we consider the case $\chi=\chi_0$. By the well-known formulas for the density of square-free integers,
\begin{equation}\label{eq:fstarc=02}
		\CU_{1,q_2,\bc}\left(\chi_0;Y\right)=\sum_{\substack{q_1\leqslant Y\\ (q_1,m\Omega)=1}}\mu(q_1)^2=\frac{6}{\pi^2}\left(\sum_{d\mid m\Omega}\frac{\mu(d)}{d}\right)Y+O_\varepsilon\left(m^\varepsilon Y^\frac{1}{2}\right).
\end{equation}
Let $\CF_{\bc}^{(2)}(\beta;X)$ be as in the previous case. Then the contribution of the non-principal $\chi\bmod q_2L^2$ together with the error term of \eqref{eq:fstarc=02} can be bounded  by, using \eqref{eq:CAbd},
\begin{align*}
	&m^\varepsilon \sum_{\substack{q_2\leqslant X^\beta\\ q_2\mid(m\Omega)^\infty}}\frac{\phi(q_2L^2)}{q_2^{2}}\sup_{\chi\bmod q_2L^2}\left|\widehat{\CA}_{q_2}(\chi;\bc)\right| \left(\frac{X}{q_2}\right)^{\frac{1}{2}+\varepsilon}q_2^{\frac{3}{16}+\varepsilon} \\ \ll & X^{\frac{1}{2}+\varepsilon}\sum_{\substack{q_2\leqslant X^\beta\\ q_2\mid(m\Omega)^\infty}}q_2^{\frac{19}{16}}\ll_{\varepsilon}  m^\varepsilon X^{\frac{1}{2}+\frac{19}{16}\beta+\varepsilon}.
\end{align*}
 So in the case $F^*(\bc)=0$,
	\begin{align*}
		\CF_{\bc}^{(2)}(\beta;X)&=\frac{6}{\pi^2}\left(\sum_{d\mid m\Omega}\frac{\mu(d)}{d}\right)X\sum_{\substack{q_2\leqslant X^\beta\\ q_2\mid(m\Omega)^\infty}}\frac{\widehat{\CA}_{q_2}(\chi_0;\bc)}{q_2^3}+O_\varepsilon\left(m^\varepsilon X^{\frac{1}{2}+\frac{19}{16}\beta+\varepsilon}\right)\\ &=\eta(\bc)X+O_\varepsilon\left(m^\varepsilon(X^{1-\frac{\beta}{2}+\varepsilon}+X^{\frac{1}{2}+\frac{19}{16}\beta+\varepsilon})\right),
	\end{align*} where we define for $\bc$ with $F^*(\bc)=0$,
	\begin{equation}\label{eq:etac1}
		\eta(\bc):=\frac{6}{\pi^2}\left(\sum_{d\mid m\Omega}\frac{\mu(d)}{d}\right)\left(\sum_{\substack{u\mid (m\Omega)^{\infty}}}\frac{\widehat{\CA}_{u}(\chi_0;\bc)}{u^3}\right),
	\end{equation}
	and $$\sum_{\substack{q_2>X^\beta\\ q_2\mid(m\Omega)^\infty}}\frac{\widehat{\CA}_{q_2}(\chi_0;\bc)}{q_2^3}\ll \sum_{\substack{q_2> X^\beta\\ q_2\mid(m\Omega)^\infty}}q_2^{-\frac{1}{2}}\ll_{\varepsilon} m^\varepsilon X^{-\frac{1}{2}\beta+\varepsilon}.$$
	Similarly $\eta(\bc)$ satisfies $$\eta(\bc)\ll \sum_{u\mid (m\Omega)^\infty}u^{-\frac{1}{2}}\ll \Upsilon_{\frac{1}{2}}(m).$$
	
	\subsection{} 
	We now prove a more sophisticated decomposition for $\CF_{\bc}(X)$ \eqref{eq:CFc}, which is slightly different from \eqref{eq:multiplicativity}.
	Let $q\in\BN$ be such that $\gcd(q,\Omega)=1$ and $q_{m(\Omega)}$ square-free. We let $q_\flat:=q/q_{m(\Omega)}$ so that $\gcd(q_\flat,m\Omega)=1$. For any $l\mid (m\Omega)^\infty$ with $\gcd(q,l)=1$, we recall \eqref{eq:CT}, and by \eqref{eq:multiplicativity} we have $$	\CT_q^1(l,\bc)=\CT_{q_{m(\Omega)}}^{q_\flat}(l,\bc) \CT_{q_\flat}^{q_{m(\Omega)}}(l,\bc).$$ The first factor is a Gauss sum:
	\begin{align*}
		\CT_{q_{m(\Omega)}}^{q_\flat}(l,\bc)&=\sum_{\substack{a\bmod q_{m(\Omega)}\\ (a,q_{m(\Omega)})=1}}\left(\frac{a}{q_{m(\Omega)}}\right)\e_{q_{m(\Omega)}} \left(-\overline{4\Delta_F aq_\flat}(lL^2)^2 F^*(\bc)\right)\\ &=\left(\frac{-\Delta_F q_\flat F^*(\bc)} {q_{m(\Omega)}}\right)\iota_{q_{m(\Omega)}},
	\end{align*} by e.g. \cite[(3.12)]{Iwaniec-Kolwalski}. The second one is a Salié sum, and by \eqref{eq:rSalie},
	\begin{align*}
		\CT_{q_\flat}^{q_{m(\Omega)}}(l,\bc)&=\sum_{\substack{b\bmod q_\flat\\ (b,q_\flat)=1}}\left(\frac{b}{q_\flat}\right)\e_{q_\flat} \left(-am\overline{q_{m(\Omega)}}-\overline{4\Delta_F aq_{m(\Omega)}}(lL^2)^2 F^*(\bc)\right)\\ &=\left(\frac{-mq_{m(\Omega)}}{q_\flat}\right)\iota_{q_\flat}\sum_{\substack{u\bmod q_\flat\\ G_{q_{m(\Omega)},l,\bc}\equiv 0\bmod q_\flat}}\e_{q_\flat}(u).
	\end{align*} 
	On the other hand (see e.g. \cite[(3.38)]{Iwaniec-Kolwalski}), 
	$$\iota_q=\iota_{q_{m(\Omega)}}\iota_{q_\flat}\left(\frac{q_{m(\Omega)}}{q_\flat}\right)\left(\frac{q_\flat}{q_{m(\Omega)}}\right).$$
	Hence by Proposition \ref{prop:S1value}, under our assumption that $-m\Delta_F=\square$, \begin{equation}\label{eq:decompS1}
		\e_{q}\left(\overline{lL^2}\blambda_m\cdot \bc\right)\widehat{S}_{ql}^{(1)}(\bc)=q^2\left(\frac{-F^*(\bc)}{q_{m(\Omega)}}\right)\sum_{\substack{u\bmod q_\flat\\ G_{q_{m(\Omega)},l,\bc}\equiv 0\bmod q_\flat}}\e_{q_\flat}(u).
	\end{equation}
	Now we are ready to derive our second decomposition
	\begin{equation}\label{eq:decomp2}
		\small\begin{split}
			\CF_{\bc}(X)
			 &=\sum_{\substack{l\leqslant X\\ l\mid (m\Omega)^\infty\\ l_{m(\Omega)}\square-\text{full}}}\sum_{\chi\bmod lL^2}\frac{\widehat{\CA}_{l}(\chi;\bc)}{l^2}\sum_{\substack{q_2\leqslant \frac{X}{l},(q_2,l)=1\\ q_2\mid m(\Omega)^\infty\\ q_2\square-\text{free}}}\left(\frac{-F^*(\bc)}{q_2}\right)\CU_{q_2,l,\bc}\left(\chi;\frac{X}{q_2l}\right).
		\end{split}
	\end{equation}
	\subsubsection{Case where $-F^*(\bc)\neq\square$.} 
	Let $\widetilde{\CF}_{\bc}^{(1)}(\beta;X)$ denote the $q$-sum such that both $q_{\Omega}$ and $q_{m(\Omega)}^\square$ are $\leqslant (\log X)^\beta$. Then Proposition \ref{prop:Kloosterman} gives
	$$\CF_{\bc}(X)-\widetilde{\CF}_{\bc}^{(1)}(\beta;X)\ll_{\varepsilon}|\bc|^{3+\varepsilon}X(\log X)^{1-\frac{\beta}{2}+\varepsilon}.$$
		Applying Theorem \ref{thm:Fstarcneqsq} (1) we obtain
		$$\CU_{q_2,l,\bc}\left(\chi;\frac{X}{q_2l}\right) \ll_\varepsilon |\bc|^{2+\varepsilon} \frac{X}{q_2l(\log X)^{1-\frac{\sqrt{2}}{2}-\varepsilon}}.$$
		Hence \begin{align*}
			\widetilde{\CF}_{\bc}^{(1)}(\beta;X)&\ll_\varepsilon |\bc|^{2+\varepsilon}\frac{X}{(\log X)^{1-\frac{\sqrt{2}}{2}-\varepsilon}} \sum_{\substack{l\mid (m\Omega)^\infty,l_{m(\Omega)}\square-\text{full}\\ \max(l_{m(\Omega)},l_\Omega)\leqslant(\log X)^\beta}}l^\frac{1}{2}\sum_{q_2\leqslant X}\frac{1}{q_2}\\ &\ll  |\bc|^{2+\varepsilon}X(\log X)^{\frac{\sqrt{2}}{2}+\frac{3}{2}\beta+\varepsilon}.
		\end{align*}
		We choose $\beta=\frac{1}{2}(1-\frac{\sqrt{2}}{2})$ to conclude.
		
		\subsubsection{Case where $-F^*(\bc)$ is a non-zero square.}
	For $\frac{1}{2}<\kappa_1<1$ let $$\psi(X):=\exp\left((\log X)^{\kappa_1}\right)$$  and let $\widetilde{\CF}_{\bc}^{(2)}(\kappa_1;X)$ denote the sum in \eqref{eq:CFc} restricted to the range $q_{m\Omega}\leqslant \psi(X)$. Then Proposition \ref{prop:Kloosterman} shows that
\begin{equation}\label{eq:CFbd2}
		\CF_{\bc}(X)-\widetilde{\CF}_{\bc}^{(2)}(\kappa_1;X)\ll_{\varepsilon}\Upsilon_{1-\kappa_1}(m)  |\bc|^{1+\varepsilon} X(\log X)\exp\left(\left(\frac{1}{2}-\kappa_1\right)(\log X)^{\kappa_1}\right).
\end{equation}
For $\beta>0$ let $\widetilde{\CF}_{\bc}^{(3)}(\kappa_1,\beta;X)$ be the $q$-sum satisfying both conditions $q_{m\Omega}\leqslant \psi(X)$ and $q_{m(\Omega)}^\square\leqslant(\log X)^\beta$. 
An easy modification of the proof \eqref{eq:Klsum2} of Proposition \ref{prop:Kloosterman} shows
\begin{align*}
	&\widetilde{\CF}_{\bc}^{(2)}(\kappa_1;X)-\widetilde{\CF}_{\bc}^{(3)}(\kappa_1,\beta;X)\\ \ll_{\varepsilon}&\sum_{\substack{q_2\leqslant \psi(X)\\ q_2\mid \Omega^\infty}}q_2^{\frac{1}{2}}\sum_{\substack{q_1' q_1 ''\leqslant \psi(X),q_1''>(\log X)^\beta\\ q_1'q_1 ''\mid (m(\Omega))^\infty\\ q_1'\square-\text{free},q_1''\square-\text{full}}}|\bc||q_1''|^{\varepsilon}\sum_{\substack{q_1\leqslant\frac{X}{q_1'q_1'' q_2}\\ (q_1,m\Omega)=1}}\tau(q_1)\\ \ll & |\bc|^{2+\varepsilon} X\log X\sum_{\substack{q_2\leqslant \psi(X)\\ q_2\mid \Omega^\infty}} q_2^{-\frac{1}{2}}\sum_{\substack{q_1'\leqslant \psi(X)\\q_1'\square-\text{free}}}q_1^{'-1} \sum_{\substack{(\log X)^\beta<q_1 ''\leqslant X \\ q_1''\square-\text{full} }}q_1^{''-1+\varepsilon}\\ \ll_\varepsilon & |\bc|^{2+\varepsilon} X(\log X)^{1+\kappa_1-\frac{1}{2}\beta+\varepsilon}.
	\end{align*}

	By Theorem \ref{thm:Fstarcneqsq} (2) (since both $q_2,l$ are $\ll_{\varepsilon} X^\varepsilon$)
	$$\CU_{q_2,l,\bc}\left(\chi;\frac{X}{q_2l}\right)=\frac{\Gamma(\chi;q_2,l;\bc)}{q_2l}X+O_\varepsilon\left((mX)^\varepsilon |\bc|^{1+\varepsilon}(X^{\frac{63}{64}}+m^\frac{1}{2})\right),$$ 
	with \begin{equation}\label{eq:bdgamma2}
		\Gamma(\chi;q_2,l;\bc)\ll_{\varepsilon} \frac{(l|\bc|)^\varepsilon }{|\chi_*|^{\frac{15}{32}}}\Upsilon_{1}(m).
	\end{equation}
	The big $O$-term contributes to $\widetilde{\CF}_{\bc}^{(3)}(\kappa_1,\beta;X)$ $$\ll_{\varepsilon} (m X)^\varepsilon|\bc|^{1+\varepsilon}(X^{\frac{63}{64}}+m^\frac{1}{2}).$$
	By Proposition \ref{prop:CAq2} and on using \eqref{eq:bdgamma2}, the contribution of the main term of $\CU_{1,q_2,\bc}\left(\chi;\frac{X}{q_1q_2}\right)$ to $\widetilde{\CF}_{\bc}^{(3)}(\kappa_1,\beta;X)$ is, analogously to \eqref{eq:chisum},
	\begin{align*}
		\leqslant & |\bc|^\varepsilon X\Upsilon_{1}(m)\sum_{\substack{l\mid (m\Omega)^\infty, l_{m(\Omega)}\square-\text{full}\\ l_{\Omega}\leqslant\psi(X),l_{m(\Omega)}\leqslant(\log X)^\beta}}\sum_{\chi\bmod lL^2}\frac{\left|\widehat{\CA}_{l}(\chi;\bc)\right|}{|\chi_*|^{\frac{15}{32}}l^{3-\varepsilon}}\sum_{\substack{q_2\square-\text{free}\\ q_2\leqslant\psi(X)}}\frac{1}{q_2}\\ \ll_{\varepsilon}&|\bc|^\varepsilon X(\log X)^{\kappa_1}\Upsilon_{1}(m)\sum_{\substack{l\mid (m\Omega)^\infty, l_{m(\Omega)}\square-\text{full}\\ l_{\Omega}\leqslant\psi(X),l_{m(\Omega)}\leqslant(\log X)^\beta}}l^{-\frac{5}{16}+\varepsilon}\sum_{e\mid lL^2}e^{\frac{9}{32}} \\ \ll_{\varepsilon}& |\bc|^\varepsilon X(\log X)^{\kappa_1}\Upsilon_{1}(m)\sum_{\substack{l_1\mid\Omega^\infty,l_1\leqslant\psi(X)\\ l_2\square-\text{full},l_2\leqslant (\log X)^\beta}} (l_1l_2)^{-\frac{1}{32}+\varepsilon}\\ \ll_\varepsilon & |\bc|^\varepsilon X(\log X)^{\kappa_1+\frac{15}{32}\beta}\Upsilon_{1}(m).
	\end{align*}
	We now take $$\kappa_1:=\frac{33}{64},\quad \beta:=\frac{32}{31}$$ to conclude.
	\subsubsection{Case where $F^*(\bc)=0$.}
	In  view of \eqref{eq:decompS1}, $\widehat{S}_{q}(\bc)\neq 0$ only if $q_{m(\Omega)}=q_{m(\Omega)}^\square$. 
	Let $\widetilde{\CF}_{\bc}^{(2)}(\kappa_1;X)$  be as in the previous case.
	According to the computations \eqref{eq:fstarc=01} \eqref{eq:fstarc=02}, the non-principal characters contribute to $\widetilde{\CF}_{\bc}^{(2)}(\kappa_1;X)$
	$$\sum_{\substack{\substack{l\mid (m\Omega)^\infty,l\leqslant \psi(X)\\ l_{m(\Omega)}\square-\text{full}}}}\sum_{\substack{\chi\bmod lL^2\\ \chi\neq\chi_0}}\frac{\left|\widehat{\CA}_{l}(\chi;\bc)\CU_{1,l,\bc}\left(\chi;\frac{X}{l}\right)\right|}{l^2}\ll m^\varepsilon X^{\frac{1}{2}+\varepsilon}.$$
	While the principal one contributes
	\begin{align*}
		\sum_{\substack{\substack{l\mid (m\Omega)^\infty,l\leqslant \psi(X)\\ l_{m(\Omega)}\square-\text{full}}}}\frac{\left|\widehat{\CA}_{l}(\chi_0;\bc)\CU_{1,l,\bc}\left(\chi_0;\frac{X}{l}\right)\right|}{l^2}\ll X\sum_{\substack{l\leqslant \psi(X)\\ l\square-\text{full}}}\frac{1}{l^\frac{1}{2}}\ll X(\log X)^{\kappa_1}.
	\end{align*}
We may take $\kappa_1:=\frac{3}{4}$ to conclude.
So the proof of Theorem \ref{thm:sumsaliesum} is completed. \qed
\section{Proof of the main theorems}

Assuming appropriate growth conditions on $m$, our goal in this section is to prove Theorems \ref{thm:mainterm} and \ref{thm:mainsecondary}. 

We need to treat separately the contribution from the Poisson variable $\bc=\boldsymbol{0}$ and $\bc\neq \boldsymbol{0}$.
\subsection{The contribution from $\bc\neq\boldsymbol{0}$}
 Without restriction on $\omega(m)$, we establish the following:
 \begin{theorem}\label{thm:csumneq0noomegem}
Assuming \eqref{eq:mB2theta}, we have 	$$\underset{\substack{\bc\in\BZ^{3}\setminus\boldsymbol{0},q\ll Q}}{\sum\sum}\frac{\widehat{S}_{q}(\bc)\widehat{I}_{q}(w;\bc)}{(qL)^3}=O\left(B^3(\log B)^{\frac{1983}{1984}}\right).$$
 \end{theorem}
 
To describe our result assuming $\omega(m)=O(1)$, let us define
 \begin{equation}\label{eq:CB}
	\CJ(\bc):=\int_{0}^\infty\frac{\widehat{\CI}_{r}\left(w;\frac{\bc}{L}\right)}{r}\operatorname{d}r=\int_{0}^\infty \int_{\BR^3}w(\bt)h\left(r,F(\bt)\right)\e_r\left(-\frac{1}{L}\bc\cdot\bt\right) \frac{\operatorname{d}r\operatorname{d}\bt}{r},
\end{equation}
 and \begin{equation}\label{eq:CK}
 	\CK:=\frac{1}{L^4}\sum_{\substack{\bc\in\BZ^3\setminus\boldsymbol{0}\\ -F^*(\bc)=\square}}\eta(\bc)\CJ(\bc),
 \end{equation}
where $\eta(\bc)$ given in Theorem \ref{thm:sumsaliesum} (2). 
We also recall our choice of $Q$ \eqref{eq:Q}.
\begin{theorem}\label{thm:csumneq0}
	Assuming \eqref{eq:mB2theta} and \eqref{eq:omegam}, we have
	$$\underset{\substack{\bc\in\BZ^{3}\setminus\boldsymbol{0},q\ll Q}}{\sum\sum}\frac{\widehat{S}_{q}(\bc)\widehat{I}_{q}(w;\bc)}{(qL)^3}=\frac{B^3}{L^2}\CK +O \left(\frac{B^3}{(\log B)^{\frac{1}{4}(1-\frac{\sqrt{2}}{2})}}\right).$$
	Moreover
\begin{equation}\label{eq:CKbd}
		\CK=O(1).
\end{equation}
\end{theorem}

\subsubsection{Preliminary manipulation}
By \eqref{eq:IJ}, the summand equals $\left(\frac{B}{L^2}\right)^3 \CG^{*}(q;\bc)$, where
\begin{equation}\label{eq:CG}
	\CG^{*}(q;\bc):=\frac{\e_{qL^2}(\bc\cdot\blambda_m)\widehat{S}_{q}(\bc)\widehat{\CI}_{\frac{q}{Q}}^{*}\left(w;\frac{\bc}{L}\right)}{q^3}.
\end{equation}
We shall first compare the sum above with $\sum_{q\ll Q}\CG(q;\bc)$ where $\CG(q;\bc)$ is defined analogously to $\CG^{*}(q;\bc)$  but with $\widehat{\CI}_{\frac{q}{Q}}^{*}\left(w;\frac{\bc}{L}\right)$ replaced by $\widehat{\CI}_{\frac{q}{Q}}\left(w;\frac{\bc}{L}\right)$ \eqref{eq:CIusual}.
\begin{lemma}\label{le:compareCGBCG}
	Assuming \eqref{eq:mB2theta}, uniformly for any $\bc\in\BZ^3\setminus\boldsymbol{0}$, we have
	$$\sum_{q\ll Q}\CG^{*}(q;\bc)=\sum_{q\ll Q}\CG(q;\bc)+O_\varepsilon( B^{-\frac{1}{5}\theta+\varepsilon}).$$
\end{lemma}
\begin{proof}
	By Proposition \ref{prop:Kloosterman} with partial summation (with no condition on $q_{m\Omega}$), and on using the harder estimate for $\widehat{\CI}_{\frac{q}{Q}}^{*}\left(w;\frac{\bc}{L}\right)$ in Lemma \ref{le:Isimplehard}, we have, for any $\delta>0$ 
	\begin{equation}\label{eq:smallrange}
		\begin{split}
		\sum_{q\leqslant B^\delta}\CG^{*}(q;\bc)&\ll_{\varepsilon} \sum_{q\leqslant B^\delta}\frac{\left|\widehat{S}_q(\bc)\right|}{q^{3}}\left(\frac{q/Q}{|\bc|}\right)^{\frac{1}{2}-\varepsilon}\\
		&\ll_\varepsilon (|\bc|Q)^{-\frac{1}{2}+\varepsilon}\sum_{q\leqslant B^\delta}\frac{\left|\widehat{S}_q(\bc)\right|}{q^{\frac{5}{2}}}\ll_{\varepsilon} |\bc|^{-\frac{1}{2}+\varepsilon}B^{-\frac{1}{2}+\frac{1}{2}\delta+\varepsilon},
		\end{split}
	\end{equation} and the same bound also holds for $\sum_{q\leqslant B^\delta}\CG(q;\bc)$.
	
	Now by Lemma \ref{le:CIBCI} and again by Proposition \ref{prop:Kloosterman} with partial summation, 
	\begin{align*}
		\sum_{B^\delta<q\ll Q}\left(\CG^{*}(q;\bc)-\CG(q;\bc)\right)&\ll \sum_{B^\delta<q\ll Q} \frac{\left|\widehat{S}_q(\bc)\right|}{q^3}\left|\widehat{\CI}_{\frac{q}{Q}}\left(w;\frac{\bc}{L}\right)-\widehat{\CI}_{\frac{q}{Q}}^{*}\left(w;\frac{\bc}{L}\right)\right|\\ &\ll B^{2-\theta} \sum_{B^\delta<q\ll Q} \frac{\left|\widehat{S}_q(\bc)\right|}{q^5}
		\ll_\varepsilon B^{2-2\delta-\theta+\varepsilon}.
	\end{align*}
	On choosing $\delta=1-\frac{2}{5}\theta$, we therefore conclude the proof of Lemma \ref{le:compareCGBCG}.
\end{proof}

\subsubsection{Proof of Theorem \ref{thm:csumneq0noomegem}}
We first prove the following elementary (non-optimal) estimate for the maximal order of \eqref{eq:upsilon}.
\begin{lemma}\label{le:upsilon}
	For $0<\kappa<1$, there exists $C_\kappa>0$ such that uniformly for any $n$, 
	$$\Upsilon_{\kappa}\ll \exp\left(C_\kappa\frac{(\log n)^{1-\kappa}}{\log\log n}\right).$$
\end{lemma}
\begin{proof}
	Let $q(n):=\lfloor A\omega(n)\log \omega(n)\rfloor$, where $A>0$ is sufficiently large and fixed, such that for any $n$, $$\omega(n)\leqslant \#\{p:p\leqslant q(n)\}.$$ Then by the maximal order estimate for $\omega$, there exists another $A'>0$ such that $q(n)\leqslant A'\log n$, uniformly for all $n$. Let $C_{\kappa}^\prime:=(1-2^{-\kappa})^{-1}$. Then
	\begin{align*}
		\Upsilon_{\kappa}(n)\leqslant \prod_{p\leqslant q(n)}(1-p^{-\kappa})^{-1} &\leqslant \prod_{p\leqslant q(n)}\left(1+\frac{C_{\kappa}^\prime}{p^{\kappa}}\right)\\ &\leqslant \exp\left(\sum_{p\leqslant q(n)}\frac{C_{\kappa}^\prime}{p^\kappa}\right)\\ &\leqslant \exp\left(C_\kappa\frac{q(n)^{1-\kappa}}{\log (q(n))}\right),
	\end{align*} for another uniform $C_\kappa>0$.
\end{proof}
\begin{proposition}\label{prop:qsumcneq0}
	Assuming \eqref{eq:mB2theta}, for any $\bc\in\BZ^3\setminus\boldsymbol{0}$, we have
	$$\sum_{q\ll Q}\CG(q;\bc)\ll_{N,\varepsilon} (B|\bc|)^{-\frac{1}{2}+\varepsilon}+ |\bc|^{-N}(\log B)^{\frac{1983}{1984}}.$$
\end{proposition}
\begin{proof}
	Let $0<\alpha<1$ be fixed and let $$\CF_{\bc}(\alpha;t):=\sum_{q\leqslant t}\frac{\e_{qL^2}(\bc\cdot\blambda_m)\widehat{S}_{q}(\bc)}{q^{3-\alpha}}.$$ Then by partial summation, the sum to be estimated equals
	$$-\int_{1}^{cQ}\CF_{\bc}(\alpha;t)\frac{\partial}{\partial t}\left(\frac{\widehat{\CI}_{\frac{t}{Q}}\left(w;\frac{\bc}{L}\right)}{t^\alpha}\right)\operatorname{d}t+\left[\CF_{\bc}(\alpha;t)\frac{\widehat{\CI}_{\frac{t}{Q}}\left(w;\frac{\bc}{L}\right)}{t^\alpha}\right]_{1}^{cQ},$$
where $c=O(1)$ such that $\widehat{\CI}_{r}(w;\bb)$ vanishes for $r\geqslant c$. Then the variation term is, by the ``harder estimate'' in Lemma \ref{le:Isimplehard}, $O_\varepsilon((Q|\bc|)^{-\frac{1}{2}+\varepsilon})$. 

	We are now concerned with the integral, which is clearly
	\begin{equation}\label{eq:qsumstep}
		\leqslant \left(\sup_{T\ll Q}\left|\CF_{\bc}(\alpha;T)\right|\right)\int_{1}^{cQ}\left|\frac{\partial }{\partial t}\left(\frac{\widehat{\CI}_{\frac{t}{Q}}\left(w;\frac{\bc}{L}\right)}{t^\alpha}\right)\right|\operatorname{d}t,
	\end{equation} 
	By Theorem \ref{thm:sumsaliesum} with partial summation (more precisely if $-F^*(\bc)\neq\square$ we apply Theorem \ref{thm:sumsaliesum} (1) on using the second bound, if $-F^*(\bc)=\square$ we apply the second bound of Theorem \ref{thm:sumsaliesum} (2)), the factor on the left with a $q$-sum is 
	\begin{align*}
		\ll_{\varepsilon_1} &|\bc|^{3+\varepsilon_1} Q^\alpha\left(\Upsilon_{\frac{31}{64}}(m)(\log Q)
		\exp\left(-\frac{1}{64}(\log Q)^{\frac{33}{64}}\right)+(\log Q)^{\frac{1983}{1984}}\right)\\ \ll &|\bc|^{3+\varepsilon_1}Q^\alpha (\log B)^{\frac{1983}{1984}},
	\end{align*} since by Lemma \ref{le:upsilon}, $$\Upsilon_{\frac{31}{64}}(m)\ll \exp\left(C_{\frac{31}{64}}\frac{(\log m)^\frac{33}{64}}{\log\log m}\right)\ll \exp\left(C_{\frac{31}{64}}\frac{(\log Q)^\frac{33}{64}}{\log\log B}\right).$$
	As for the integral on the right, the change of variable $t\mapsto rQ$ yields that it is
	\begin{align*}
		&= Q^{-\alpha}\int_{Q^{-1}}^{c}\left|\frac{\partial }{\partial r}\left(\frac{\widehat{\CI}_{r}\left(w;\frac{\bc}{L}\right)}{r^\alpha}\right)\right|\operatorname{d}r\\ &\ll Q^{-\alpha}\int_{0}^{c}\left(r^{-\alpha}\left|\frac{\partial \widehat{\CI}_{r}\left(w;\frac{\bc}{L}\right)}{\partial r}\right|+r^{-\alpha-1}\left|\widehat{\CI}_{r}\left(w;\frac{\bc}{L}\right)\right|\right)\operatorname{d}r,
	\end{align*} and by the ``easier'' and the ``harder'' estimates in Lemma \ref{le:Isimplehard}, it is
\begin{align*}
	&\ll_{\varepsilon_0,N} Q^{-\alpha}\int_{0}^{c}\min\left(r^{-(2+\alpha)}|\bc|^{-N},r^{-(\frac{1}{2}+\alpha+\varepsilon_0)}|\bc|^{-\frac{1}{2}+\varepsilon_0}\right)\operatorname{d}r\\ &\ll Q^{-\alpha}\left(|\bc|^{-N}\int_{|\bc|^{-\frac{2N}{3}}}^{c}r^{-(2+\alpha)}\operatorname{d}r+|\bc|^{-\frac{1}{2}+\varepsilon_0} \int_{0}^{|\bc|^{-\frac{2N}{3}}}r^{-(\frac{1}{2}+\alpha+\varepsilon_0)}\operatorname{d}r\right)\\&\ll_{\varepsilon_0,N} Q^{-\alpha}|\bc|^{-\frac{N}{3}(1-2\alpha-\varepsilon_0)}.
\end{align*}
On fixing $0<\alpha<\frac{1}{2}$, $\varepsilon_0,\varepsilon_1$ small enough, and on redefining  $N$, we obtain that the initial integral is 
$$\ll_{N} |\bc|^{-N}(\log B)^{\frac{1983}{1984}}.$$
The proof is thus finished.
\end{proof}
\begin{proof}[Completion of the proof of Theorem \ref{thm:csumneq0noomegem}]
	Using the simple estimate for $\widehat{\CI}_{r}^{*}(w;\bb)$ in Lemma \ref{le:Isimplehard}, we can reduce the $\bc$-sum to those with $|\bc|\leqslant B^\varepsilon$, with an arbitrarily large power-saving. Then by Lemma \ref{le:compareCGBCG} and Proposition \ref{prop:qsumcneq0}, 
	\begin{align*}
		\underset{\substack{\bc\in\BZ^{3}\setminus\boldsymbol{0},|\bc|\leqslant B^\varepsilon\\q\ll Q}}{\sum\sum}\CG^{*}(q;\bc)&\ll_{N,\varepsilon} \left(B^{-\frac{1}{5}\theta+\varepsilon}+B^{-\frac{1}{2}+\varepsilon}\right)\sum_{0<|\bc|\leqslant B^\varepsilon}1+(\log B)^{\frac{1983}{1984}}\sum_{\bc\neq\boldsymbol{0}}|\bc|^{-N}\\ &\ll (\log B)^{\frac{1983}{1984}},
	\end{align*} by taking $N$ large enough.
\end{proof}

\subsubsection{Proof of Theorem \ref{thm:csumneq0}}
\begin{proposition}\label{prop:qsumFstarcneqsq}
	Assuming \eqref{eq:mB2theta} and \eqref{eq:omegam}, for any $\bc\in\BZ^3$ with $-F^*(\bc)\neq \square$, we have
	$$\sum_{q\ll Q}\CG(q;\bc)\ll_{N,\varepsilon} (B|\bc|)^{-\frac{1}{2}+\varepsilon}+ |\bc|^{-N}(\log B)^{-\frac{1}{4}(1-\frac{\sqrt{2}}{2})}.$$
\end{proposition}
\begin{proof}
	This requires only minor modification of the proof of Proposition \ref{prop:qsumcneq0}. Indeed by Theorem \ref{thm:sumsaliesum} (1) with partial summation, the $q$-sum in \eqref{eq:qsumstep} is 
	$$\ll_{\varepsilon}|\bc|^{3+\varepsilon}\frac{Q^\alpha}{(\log Q)^{\frac{1}{4}(1-\frac{\sqrt{2}}{2})}},$$ since $\omega(m)=O(1)$. Following the remaining step yields the desired bound.
\end{proof}
\begin{proposition}\label{prop:qsumFstarc=sq}
	Assuming \eqref{eq:mB2theta}, for $\bc\in\BZ^3\setminus\boldsymbol{0}$ such that $-F^*(\bc)=\square$, we have
	$$\sum_{q\ll Q}\CG(q;\bc)=\eta(\bc)\CJ(\bc)+O_\varepsilon\left(|\bc|^{\frac{1}{2}+\varepsilon}B^{-\min\left(\frac{1}{22400},\frac{1}{320}\theta\right)+\varepsilon}\right),$$ where $\CJ(\bc)$ is defined by \eqref{eq:CB}.
\end{proposition}

\begin{proof}
	Let $\vartheta>0$ be chosen later on. In view of \eqref{eq:smallrange}, we may restrict ourselves to the range $Q^{1-\vartheta}\ll q\ll Q$, upon adding an error term $O_\varepsilon(|\bc|^{\frac{1}{2}+\varepsilon} B^{-\frac{1}{2}\vartheta+\varepsilon})$. By partial summation, we have (recall \eqref{eq:CFc})
	\begin{align*}
		\sum_{Q^{1-\vartheta}<q\ll  Q}\CG(q;\bc)=-\int_{Q^{1-\vartheta}}^{cQ}\left(\CF_{\bc}(t)\frac{\partial }{\partial t}\left(\frac{\widehat{\CI}_{\frac{t}{Q}}\left(w;\frac{\bc}{L}\right)}{t}\right) \right)\operatorname{d}t +
		\left[\CF_{\bc}(t)\frac{\widehat{\CI}_{\frac{t}{Q}}\left(w;\frac{\bc}{L}\right)}{t}\right]_{Q^{1-\vartheta}}^{cQ}. 
	\end{align*} We recall that $c>0$ is fixed so that the function $r\mapsto\widehat{\CI}_{r}(w;\bb)$ is supported in $(0,c)$. 
By the ``harder estimate'' in Lemma \ref{le:Isimplehard}, \begin{equation}\label{eq:harderCI}
	\widehat{\CI}_{Q^{-\vartheta}}\left(w;\frac{\bc}{L}\right)\ll_{\varepsilon} Q^{-\frac{1}{2}\vartheta+\varepsilon}.
\end{equation} Being similar to the proof of Proposition \ref{prop:qsumcneq0},   the variation term equals $$\CF_{\bc}(Q^{1-\vartheta})\frac{\widehat{\CI}_{Q^{-\vartheta}}\left(w;\frac{\bc}{L}\right)}{Q^{1-\vartheta}}\ll_{\varepsilon} |\bc|^{1+\varepsilon}Q^{-\frac{1}{2}\vartheta+\varepsilon},$$ by \eqref{eq:harderCI} and the upper bound estimate in Theorem \ref{thm:sumsaliesum} (2).

By Theorem \ref{thm:sumsaliesum} (2), the main term with $\eta(\bc)$ takes the form
$$-\eta(\bc)\int_{Q^{1-\vartheta}}^{cQ}t\frac{\partial }{\partial t}\left(\frac{\widehat{\CI}_{\frac{t}{Q}}\left(w;\frac{\bc}{L}\right)}{t}\right)\operatorname{d}t= -\eta(\bc)\int_{Q^{1-\vartheta}}^{cQ}t\frac{\partial }{\partial t}\left(\frac{\widehat{\CI}_{\frac{t}{Q}}\left(w;\frac{\bc}{L}\right)}{t}\right)\operatorname{d}t$$ Bu partial summation, the integral equals $$\int_{Q^{1-\vartheta}}^{cQ} \frac{\widehat{\CI}_{\frac{t}{Q}}\left(w;\frac{\bc}{L}\right)}{t}\operatorname{d}t+O_\varepsilon(Q^{-\frac{1}{2}\vartheta+\varepsilon}).$$ Next the change of variable $t\mapsto r:=\frac{t}{Q}$ yields
\begin{align*} \int_{Q^{1-\vartheta}}^{cQ} \frac{\widehat{\CI}_{\frac{t}{Q}}\left(w;\frac{\bc}{L}\right)}{t}\operatorname{d}t = \int_{Q^{-\vartheta}}^{c} \frac{\widehat{\CI}_{r}\left(w;\frac{\bc}{L}\right)}{r}\operatorname{d}r=\CJ(\bc)+O_\varepsilon(Q^{-\frac{1}{2}\vartheta+\varepsilon}),
\end{align*} 
where we use again \eqref{eq:harderCI} to extend the integral around $0$. 

On using again the ``harder estimate'' in Lemma \ref{le:Isimplehard} and on taking $\beta=\min\left(\frac{1}{700},\frac{1}{10}\theta\right)$, the error term arising from Theorem \ref{thm:sumsaliesum} (2) takes the form
\begin{align*}
	&B^\varepsilon|\bc|^{1+\varepsilon}\int_{Q^{1-\vartheta}}^{cQ}\left(t^{\frac{127}{128}}+B^{1-\frac{1}{2}\theta}t^{\frac{1}{4}\theta}+t^{1-\min\left(\frac{1}{22400},\frac{1}{320}\theta\right)}\right)\left|\frac{\partial }{\partial t}\left(\frac{\widehat{\CI}_{\frac{t}{Q}}\left(w;\frac{\bc}{L}\right)}{t}\right)\right|\operatorname{d}t\\ \ll_{\varepsilon} &B^\varepsilon |\bc|^{\frac{1}{2}+\varepsilon}\int_{Q^{-\vartheta}}^{c} \left(B^{-\frac{1}{128}}r^{-\frac{65}{128}}+B^{-\frac{1}{2}\theta}r^{\frac{1}{4}\theta-\frac{3}{2}}+B^{-\min\left(\frac{1}{22400},\frac{1}{320}\theta\right)}r^{-\frac{1}{2}-\min\left(\frac{1}{22400},\frac{1}{320}\theta\right)}\right)\operatorname{d}r \\ \ll &  B^\varepsilon|\bc|^{\frac{1}{2}+\varepsilon} \left(B^{-\frac{1}{2}\theta+\frac{1}{2}\vartheta}+B^{-\min\left(\frac{1}{22400},\frac{1}{320}\theta\right)}\right).
\end{align*}
It remains to choose $\vartheta=\frac{1}{2}\theta$ to complete the proof.
\end{proof}

\begin{proof}[Completion of proof of Theorem \ref{thm:csumneq0}]
	As before we are reduced to the range $|\bc|\leqslant B^\varepsilon$.
Combining Lemma \ref{le:compareCGBCG} and Propositions  \ref{prop:qsumFstarcneqsq}  \& \ref{prop:qsumFstarc=sq} gives
\begin{multline*}
	\underset{\substack{\bc\in\BZ^{3}\setminus\boldsymbol{0},|\bc|\leqslant B^\varepsilon\\q\ll Q}}{\sum\sum}\CG^{*}(q;\bc)=\sum_{\substack{0<|\bc|\leqslant B^\varepsilon\\ -F^*(\bc)=\square}}\eta(\bc)\CJ(\bc)\\+O_{N,\varepsilon}\left((\log B)^{-\frac{1}{4}(1-\frac{\sqrt{2}}{2})}\sum_{\bc\neq\boldsymbol{0}}|\bc|^{-N}+\left(B^{-\frac{1}{5}\theta+\varepsilon}+B^{-\min\left(\frac{1}{22400},\frac{1}{320}\theta\right)+\varepsilon}\right)\sum_{0<|\bc|\leqslant B^\varepsilon}|\bc|^{1+\varepsilon}\right).
\end{multline*} The big $O$-term is clearly $O\left((\log B)^{^{-\frac{1}{4}(1-\frac{\sqrt{2}}{2})}}\right)$. As for the main term, using again the simpler and harder estimates for $\widehat{\CI}_{r}(w;\bb)$ together in Lemma \ref{le:Isimplehard} gives
$$\CJ(\bc)\ll_{N,\varepsilon} |\bc|^{-\frac{1}{2}+\varepsilon}\int_{0}^{|\bc|^{-\frac{2}{3}N}}r^{-\frac{1}{2}}\operatorname{d}r+|\bc|^{-N}\int_{|\bc|^{-\frac{2}{3}N}}^{\infty}r^{-2}\operatorname{d}r\ll |\bc|^{-\frac{1}{3}N},$$
which allows to extend the $\bc$-sum to infinity with an arbitrarily large power-saving. 
Together with the bound for $\eta(\bc)$ in Theorem \ref{thm:sumsaliesum} we obtain \eqref{eq:CKbd}.
\end{proof}
\subsection{The contribution from $\bc=\boldsymbol{0}$}
\begin{theorem}\label{thm:c=0}
	Assuming \eqref{eq:mB2theta}, there exists $b=b(w;m;L,\bGamma_m)\in\BC$ such that 
	$$\sum_{q=1}^\infty\frac{\widehat{S}_{q}(\boldsymbol{0})\widehat{I}_{q}(w;\boldsymbol{0})}{(qL)^3}=\frac{B^3}{L^2}\left(\CI(w)\widehat{\mathfrak{S}}(\CW_m;L,\bGamma_m)\log B+b\right)+O_\varepsilon(B^{\frac{5}{2}+\varepsilon}+B^{3-\theta+\varepsilon}),$$ where we recall \eqref{eq:CIw} and \eqref{eq:singser}.
	Moreover we have $$b\ll \widehat{\mathfrak{S}}(\CW_m;L,\bGamma_m).$$
\end{theorem}
Our proof of Theorem \ref{thm:c=0} roughly follows \cite[\S13]{H-Bdelta} but requires some variation.
\begin{proposition}\label{prop:ssumc=0}
	We have
	$$\sum_{q\leqslant X}\widehat{S}_{q}(\boldsymbol{0})=\frac{1}{3}L^4\widehat{\mathfrak{S}}(\CW_m;L,\bGamma_m)X^3+O_\varepsilon(m^\varepsilon X^{\frac{5}{2}+\varepsilon}),$$
	$$\sum_{q\leqslant X}\frac{\widehat{S}_{q}(\boldsymbol{0})}{q^3}=\widehat{\mathfrak{S}}(\CW_m;L,\bGamma_m)\left(L^4 \log X+\gamma\right)+O_\varepsilon(m^\varepsilon X^{-\frac{1}{2}+\varepsilon}),$$ where $\gamma$ is the Euler constant. Moreover, \begin{equation}\label{eq:frakG}
		\Upsilon_{1}(m)^{-1}\ll \widehat{\mathfrak{S}}(\CW_m;L,\bGamma_m)\ll 1.
	\end{equation}
\end{proposition}
\begin{proof}
	Let us define the formal series $$\varPi(s):=\sum_{q=1}^\infty \frac{\widehat{S}_{q}(\boldsymbol{0})}{q^s}.$$ It follows from the Chinese remainder theorem (as in the proof of \cite[Theorem 6.4]{PART1}) that $\widehat{S}_{q}(\boldsymbol{0})$ is multiplicative in $q$. By Proposition \ref{prop:S1value} and \eqref{eq:S2CS}, whenever $\gcd(q,m\Omega)=1$, we have $$\widehat{S}_{q}(\boldsymbol{0})=\iota_q^3\left(\frac{\Delta_F}{q}\right)\sum_{\substack{a\bmod q\\ (a,q)=1}}\left(\frac{a}{q}\right)\e_{q}(-am)=\iota_q^4\left(\frac{-m\Delta_F}{q}\right)\mu^2(q)=q^2\mu^2(q).$$ 
	So we have $$\varPi(s)=\zeta(s-2)\widehat{\nu}(s),$$ where \begin{equation}\label{eq:nu}
		 \begin{split}
		\widehat{\nu}(s)&:=\prod_{p<\infty}\left(1-p^{-(s-2)}\right)\sum_{t=0}^{\infty}\frac{\widehat{S}_{p^t}(\boldsymbol{0})}{p^{ts}}\\&=\prod_{p\nmid m\Omega}(1-p^{-2(s-2)})\times\prod_{p\mid m\Omega}(1-p^{-(s-2)})\sum_{t=0}^{\infty}\frac{\widehat{S}_{p^t}(\boldsymbol{0})}{p^{ts}}
	\end{split}
	\end{equation} By Proposition \ref{prop:S2}, $\widehat{\nu}(s)$ is absolutely convergent for $\Re(s)>\frac{5}{2}$.
Then an application of Perron's formula as in the proof of \cite[Lemmas 30 \& 31]{H-Bdelta} yields that (recall $\omega(m)=O(1)$)
$$\sum_{q\leqslant X}\widehat{S}_{q}(\boldsymbol{0})=\operatorname{Res}_{s=3}\left(\varPi(s)\frac{X^s}{s}\right)+O_\varepsilon(X^{\frac{5}{2}+\varepsilon}),$$ 
$$\sum_{q\leqslant X}\frac{\widehat{S}_{q}(\boldsymbol{0})}{q^3}=\operatorname{Res}_{s=0}\left(\varPi(s+3)\frac{X^s}{s}\right)+O_\varepsilon(X^{-\frac{1}{2}+\varepsilon}).$$ To compute these residues, we use the well-known Laurent expansion of the Riemann zeta function around $s=1$ to write
$$\varPi(s)=\widehat{\nu}(s)\left(\frac{1}{s-3}+\gamma+(s-3)A(s-3)\right)$$ where $A(s)$ is  holomorphic around $s=0$. We also have the expansion
$$\frac{X^s}{s}=\sum_{n=0}^{\infty}s^{n-1}\frac{(\log X)^n}{n!}.$$ Hence \begin{align*}
	\operatorname{Res}_{s=3}\left(\varPi(s)\frac{X^s}{s}\right)=\frac{1}{3}\widehat{\nu}(3)X^3,
\end{align*}
\begin{align*}
	\operatorname{Res}_{s=0}\left(\varPi(s+3)\frac{X^s}{s}\right)&=(\gamma+\log X)\widehat{\nu}(3).
\end{align*} 
Now by \cite[Theorem 6.4]{PART1}, we have that for any prime $p$, $\sum_{t=0}^{k}\frac{\widehat{S}_{p^t}(\boldsymbol{0})}{p^{3t}}$ equals
$$p^{4\operatorname{ord}_p(L)}\frac{\#\{\bxi\bmod p^{k+2\operatorname{ord}_p(L)}:\bxi\equiv \bGamma_m\bmod p^{\operatorname{ord}_p(L)},F(\bxi)\equiv m\bmod p^{k+2\operatorname{ord}_p(L)}\}}{p^{2(k+2\operatorname{ord}_p(L))}},$$ and hence
$$\sum_{t=0}^{\infty}\frac{\widehat{S}_{p^t}(\boldsymbol{0})}{p^{3t}}=p^{4\operatorname{ord}_p(L)}\sigma_p(\CW_m;L,\bGamma_m).$$ So $$\widehat{\nu}(3)=\prod_{p<\infty}\left(1-\frac{1}{p}\right)\sum_{t=0}^{\infty}\frac{\widehat{S}_{p^t}(\boldsymbol{0})}{p^{3t}}=L^4\widehat{\mathfrak{S}}(\CW_m;L,\bGamma_m).$$

We now investigate $\widehat{\mathfrak{S}}(\CW_m;L,\bGamma_m)$ more closely.
We have seen by \eqref{eq:CVvalue} that whenever $p\nmid m\Omega$, $$\sum_{t=0}^{\infty}\frac{\widehat{S}_{p^t}(\boldsymbol{0})}{p^{3t}}=1+\frac{1}{p}.$$
Now assume $p\mid m(\Omega)$. Recall \eqref{eq:CT}. We have
$\CT(p;1,\boldsymbol{0})=\sum_{\substack{a\bmod p\\(a,p)=1}}\left(\frac{a}{p}\right)=0$, whence $\widehat{S}_{p}(\boldsymbol{0})=0$. 
On the other hand, $$\CT(p^2;1,\boldsymbol{0})=\sum_{\substack{a\bmod p^2\\(a,p)=1}}\e_{p^2}(-am)=\mu\left(\frac{p^2}{(p^2,m)}\right)\frac{\phi(p^2)}{\phi\left(p^2/(p^2,m)\right)}=\begin{cases}
	\phi(p^2) &\text{ if } p^2\mid m;\\ -p &\text{ if } p\| m.
\end{cases}$$ and hence $$\widehat{S}_{p^2}(\boldsymbol{0})=\begin{cases}
p^5\left(1-\frac{1}{p}\right) &\text{ if } p^2\mid m;\\ -p^4 &\text{ if } p\| m.
\end{cases}$$ Moreover by Proposition \ref{prop:S2}, we have $\widehat{S}_{p^t}(\boldsymbol{0})\ll p^{\frac{5}{2}t}$. So
$$1-\frac{1}{p^2}+O_\varepsilon\left(\frac{1}{p^{\frac{3}{2}-\varepsilon}}\right)\leqslant\sum_{t=0}^{\infty}\frac{\widehat{S}_{p^t}(\boldsymbol{0})}{p^{3t}}\leqslant 1+\frac{1}{p}\left(
1-\frac{1}{p}\right)+O_\varepsilon\left(\frac{1}{p^{\frac{3}{2}-\varepsilon}}\right).$$
We finally conclude
$$\prod_{p \mid m}\left(1-p^{-1}\right)\asymp\prod_{p \mid m(\Omega)}\left(1-p^{-1}\right)\ll \widehat{\mathfrak{S}}(\CW_m;L,\bGamma_m)\asymp \prod_{p\nmid \Omega}\left(1-\frac{1}{p}\right)\sum_{t=0}^{\infty}\frac{\widehat{S}_{p^t}(\boldsymbol{0})}{p^{3t}}= O(1).$$
This completes the proof.
\end{proof}
\begin{proof}[Proof of Theorem \ref{thm:c=0}]
	Let $0<\delta<1$ be fixed.  We recall \eqref{eq:CG}. 
	 By using Proposition \ref{prop:Kloosterman} (with partial summation) and (the second equality of) Lemma \ref{le:CIBCI}, 
	 \begin{align*}
	 	\sum_{q\leqslant Q^{1-\delta}} \CG^{*}(q;\boldsymbol{0})-\sum_{q\leqslant Q^{1-\delta}} \CG(q;\boldsymbol{0})& =\sum_{q\ll Q}\frac{\left|\widehat{S}_{q}(\boldsymbol{0})\right|}{q^3}\left(O(B^{-\theta})+O_N\left(B^{-\delta N}\right)\right)\\ &=O_{\varepsilon,N}(B^{-\theta+\varepsilon}+B^{-\delta N+\varepsilon}).
	 \end{align*}
		Therefore in the range $q\leqslant Q^{1-\delta}$  we can  focus on the sum over $\CG(q;\boldsymbol{0})$.
	We follow the argument of \cite[p. 202--p. 203]{H-Bdelta}. 	
	By \cite[Lemma 13]{H-Bdelta} (see also \cite[Theorem 6.2]{PART1}),
	\begin{align*}
		&\sum_{q\leqslant Q^{1-\delta}} \CG(q;\boldsymbol{0})\\ =&\CI(w)\left(\sum_{q\leqslant Q^{1-\delta}}\frac{\widehat{S}_{q}(\boldsymbol{0})}{q^3}\right)+O_N\left(\sum_{q\leqslant Q^{1-\delta}}\frac{\left|\widehat{S}_{q}(\boldsymbol{0})\right|}{q^3}\left(\frac{q}{Q}\right)^N\right)\\ =&\CI(w)\left(\widehat{\mathfrak{S}}(\CW_m;L,\bGamma_m)\left(L^4(1-\delta)\log Q+\gamma\right)+O_\varepsilon(Q^{-\frac{1}{2}(1-\delta)+\varepsilon})\right)+O_{N,\varepsilon}(B^{-\delta N+\varepsilon})\\ =&\CI(w)\widehat{\mathfrak{S}}(\CW_m;L,\bGamma_m)\left(L^4(1-\delta)\log Q+\gamma\right)+O_\varepsilon(B^{-\frac{1}{2}(1-\delta)+\varepsilon}),
	\end{align*} on taking $N$ large enough.

Alternatively, write $A_t:=\sum_{q\leqslant t}\widehat{S}_q(\boldsymbol{0})$. We also have, by partial summation,
\begin{align*}
	\sum_{Q^{1-\delta}<q\leqslant cQ} \CG^{*}(q;\boldsymbol{0})&=-\int_{Q^{1-\delta}}^{cQ}A_t\frac{\partial }{\partial t}\left(\frac{\widehat{\CI}^{*}_{\frac{t}{Q}}(w;\boldsymbol{0})}{t^3}\right)\operatorname{d}t +
	\left[A_t\frac{\widehat{\CI}^{*}_{\frac{t}{Q}}(w;\boldsymbol{0})}{t^3}\right]_{Q^{1-\delta}}^{cQ}. 
\end{align*}
 By Proposition \ref{prop:ssumc=0} and the ``trivial estimate'' in Lemma \ref{le:Isimplehard}, the variation term equals $$A_{Q^{1-\delta}}\frac{\widehat{\CI}^{*}_{Q^{-\delta}}(w;\boldsymbol{0})}{Q^{3(1-\delta)}}=\frac{1}{3}L^4\widehat{\mathfrak{S}}(\CW_m;L,\bGamma_m)\widehat{\CI}^{*}_{Q^{-\delta}}(w;\boldsymbol{0})+O_\varepsilon(Q^{-\frac{1}{2}+\frac{1}{2}\delta+\varepsilon}).$$ Moreover, the term with integral equals \begin{equation}\label{eq:step0}
 	-\frac{1}{3}L^4\widehat{\mathfrak{S}}(\CW_m;L,\bGamma_m)\int_{Q^{1-\delta}}^{cQ}t^3\frac{\partial }{\partial t}\left(\frac{\widehat{\CI}^{*}_{\frac{t}{Q}}(w;\boldsymbol{0})}{t^3}\right)\operatorname{d}t,
 \end{equation}with error \begin{align*}
 \ll_\varepsilon & \int_{Q^{1-\delta}}^{cQ}t^{\frac{5}{2}+\varepsilon}\left|\frac{\partial }{\partial t}\left(\frac{\widehat{\CI}^{*}_{\frac{t}{Q}}(w;\boldsymbol{0})}{t^3}\right)\right|\operatorname{d}t\ll_{\varepsilon} Q^{-\frac{1}{2}+\frac{3}{2}\delta+\varepsilon},
\end{align*} by Lemma \ref{le:Isimplehard}. 
 Integration by parts yields that the integral in \eqref{eq:step0} equals 
 $$-\widehat{\CI}^{*}_{Q^{-\delta}}(w;\boldsymbol{0})-3\int_{Q^{1-\delta}}^{cQ}\frac{\widehat{\CI}^{*}_{\frac{t}{Q}}(w;\boldsymbol{0})}{t}\operatorname{d}t.$$ (Here we use the fact that $\widehat{\CI}_{r}^{*}(w;\boldsymbol{c})$ vanishes around $r=c$.) 
 We conclude that 
 $$\sum_{Q^{1-\delta}<q\leqslant cQ} \CG^{*}(q;\boldsymbol{0})=L^4\widehat{\mathfrak{S}}(\CW_m;L,\bGamma_m)\int_{Q^{-\delta}}^{c}\frac{\widehat{\CI}^{*}_{r}(w;\boldsymbol{0})}{r}\operatorname{d}r+O_\varepsilon(Q^{-\frac{1}{2}+\frac{3}{2}\delta+\varepsilon}),$$ once the usual change of variable $t\mapsto r:=\frac{t}{Q}$ is executed. We now use (the first equality of) Lemma \ref{le:CIBCI} to compute
 $$\int_{Q^{-\delta}}^{c}\frac{\widehat{\CI}^{*}_{r}(w;\boldsymbol{0})}{r}\operatorname{d}r=\int_{Q^{-\delta}}^{c}\frac{\widehat{\CI}_{r}(w;\boldsymbol{0})}{r}\operatorname{d}r+O(B^{-\theta+2\delta}).$$

Let us consider the function $K:\BR_{>0}\to \BC$ by
$$K(v):=\CI(w)\log v+\int_{v}^{c}\frac{\widehat{\CI}_{r}(w;\boldsymbol{0})}{r}\operatorname{d}r.$$
Then the argument in \cite[p. 202--203]{H-Bdelta} (an application of \cite[Lemma 13]{H-Bdelta}) shows that 
$$K(0):=\lim_{v\to 0^+}K(v)$$ exists and as $v\to 0^+$, $$K(v)=K(0)+O_N(v^N).$$
Consequently, $$\int_{Q^{-\delta}}^{c}\frac{\widehat{\CI}_{r}(w;\boldsymbol{0})}{r}\operatorname{d}r=\delta\CI(w)\log Q+K(0)+O_N(Q^{-N\delta}).$$
So gathering together all the estimates we obtained so far gives
\begin{align*}
	&\sum_{q\ll Q} \CG(q;\boldsymbol{0})\\=&\widehat{\mathfrak{S}}(\CW_m;L,\bGamma_m)\left(\CI(w)\left(L^4\log Q+\gamma\right)+L^4K(0)\right)+O_\varepsilon(B^{-\frac{1}{2}+\frac{3}{2}\delta+\varepsilon})+O_N(B^{-N\delta})\\ =&L^4\left(\widehat{\mathfrak{S}}(\CW_m;L,\bGamma_m)\left(\CI(w)\log B+\CI(w)\left(L^{-4}\gamma-\log L\right)+K(0)\right)\right)+O_{\varepsilon}(B^{-\frac{1}{2}+\varepsilon}),
\end{align*}
on first taking $N$ large enough depending on $\delta$, then on taking $\delta$ arbitrarily small. We conclude the proof of Theorem \ref{thm:c=0} with \begin{equation}\label{eq:b}
	b=\widehat{\mathfrak{S}}(\CW_m;L,\bGamma_m)\left(\CI(w)\left(L^{-4}\gamma-\log L\right)+K(0)\right),
\end{equation} which is clearly $\ll\widehat{\mathfrak{S}}(\CW_m;L,\bGamma_m)$.
\end{proof}

\subsection{Completion of the proof of Theorem \ref{thm:mainterm}}
Recalling our choice of $Q$ \eqref{eq:Q}, going back to \eqref{eq:poisson}, applying Theorems \ref{thm:csumneq0noomegem} and \ref{thm:c=0} gives
\begin{align*}
	\sum_{\substack{\bx\in\BZ^3: F(\bx)=m\\ \bx\equiv \bGamma_m\bmod L}}w\left(\frac{\bx}{B}\right)&=\frac{C_Q}{Q^2}\underset{\substack{\bc\in\BZ^{3},q\ll Q}}{\sum\sum}\frac{\widehat{S}_{q}(\bc)\widehat{I}_{q}(w;\bc)}{(qL)^3}\\ &= \frac{C_Q}{Q^2}\left(\sum_{q=1}^\infty\frac{\widehat{S}_{q}(\boldsymbol{0})\widehat{I}_{q}(w;\boldsymbol{0})}{(qL)^3}+\underset{\substack{\bc\in\BZ^{3}\setminus\boldsymbol{0},q\ll Q}}{\sum\sum}\frac{\widehat{S}_{q}(\bc)\widehat{I}_{q}(w;\bc)}{(qL)^3}\right)\\
	&=\CI(w)\widehat{\mathfrak{S}}(\CW_m;L,\bGamma_m) B(\log B+O(1))+O\left(B(\log B)^{\frac{1983}{1984}}\right). \end{align*}
	To get the lower bound for $\widehat{\mathfrak{S}}(\CW_m;L,\bGamma_m)$ we use \eqref{eq:frakG} and note that (see e.g. \cite[I. \S5.4 \S5.5]{Tenenbaum}) $$\Upsilon_{1}(m) \ll \log\log m. \qed$$
\subsection{Completion of the proof of Theorem \ref{thm:mainsecondary}}
Applying instead Theorems \ref{thm:csumneq0} and \ref{thm:c=0} furnishes,
\begin{align*}
	\sum_{\substack{\bx\in\BZ^3: F(\bx)=m\\ \bx\equiv \bGamma_m\bmod L}}w\left(\frac{\bx}{B}\right)=\CI(w)\widehat{\mathfrak{S}}(\CW_m;L,\bGamma_m) B\log B+aB+O\left(\frac{B}{(\log B)^{\frac{1}{4}(1-\frac{\sqrt{2}}{2})}}\right),  
\end{align*} where we define $$a:=\CK+b,$$ $\CK,b$ being defined by \eqref{eq:CK} and by \eqref{eq:b} respectively, and clearly $a=O(1)$.
\qed

\appendix
\section{Interpretation via Chambert-Loir--Tschinkel}
In \cite{CL-T}, Chambert-Loir and Tschinkel initiate a programme on counting integral points on projective varieties along with an interpretation of the asymptotic formulas.

Let us place ourselves in the simplest setting which covers the affine surface that we study before.
Let $X$ be a smooth projective split variety over $\BQ$.  Let $D\subset X$ be an effective, geometrically integral divisor such that the log-anticanonical line bundle $\omega_X(D)^{-1}$ is ample, and let $H$ be a height function associated to $\omega_X(D)^{-1}$. Let $\CU$ be a flat integral model of the open subset $U:=X\setminus D$. After \cite{CL-T}, we expect that, for $\CA\subset\CU(\widehat{\BZ})$ a non-empty adelic open subset, 
$$\#\{P\in\CA\cap X(\BQ):H(P)\leqslant B\}\sim \alpha(D)\tau_\infty(D)\tau_f(\CA)B(\log B)^{\operatorname{rank}\operatorname{Pic}(X)-1},$$ where $$\alpha(D):= \frac{1}{(\operatorname{rank}\operatorname{Pic}(X)-1)!}\int_{\mathcal{C}(X)^\vee}\exp\left(-\langle\omega_X(D)^{-1},\by\rangle\right)\operatorname{d}\by,$$ is analogous to Peyre's $\alpha$-constant (which roughly speaking computes a certain weighted volume of the dual of the pseudo-effective cone $\mathcal{C}(X)$),
$\tau_\infty(D)$ is certain real Tamagawa volume of $D(\BR)$ defined in terms of the ``residue measure'' induced by $\omega_X(D)^{-1}$, and
$$\tau_f:=\prod_{p<\infty}\left(1-\frac{1}{p}\right)^{\operatorname{rank}\operatorname{Pic}(X)-1}\tau_{(X,D),p},$$ is a finite Tamagawa measure on $\CU(\widehat{\BZ})$. The key reason why we should only see $D(\BR)$ in the real part is that integral points of $\CU$, or rational points of $X$ that are integral with respect to $D$, accumulate arbitrarily close to the boundary $D$ as the height grows. So the appearance of certain volume of $D(\BR)$ is the ``limit'' of such a behaviour.

In our case, the height is associated to the naive metric, and the surface $W_m$ can be compactified into the split quadric surface $X_m:=\BP^1\times\BP^1\subset \BP^3$ by adding the boundary divisor $D:(F=0)$, a conic isomorphic to $\BP^1$ as a $\CO(1,1)$-section of $\BP^1\times\BP^1$. We have $\operatorname{rank}\operatorname{Pic}(X_m)=2$, $\omega_{X_m}(D)^{-1}\simeq \CO(1,1)$, $\alpha(D)=1$, and $\tau_{(X_m,D),p}$ coincides with the induced Tamagawa measure on $W_m(\BQ_p)$. Finally, $\tau_\infty=2\times \tau_{D,\infty}$ on $D(\BR)$, where the real measure $\tau_{D,\infty}$ is induced by $\CO_{\BP^1}(2)$. The discussion in \cite[Remark 6.3]{PART1} shows that the singular integral $\CI(w)$ \eqref{eq:CIw} is the ``weighted'' version of the real volume $\tau_\infty(D(\BR))$.


	\section*{Acknowledgements}
	We are very grateful to Runlin Zhang for communicating the references \cite{K-K,K-K2,Oh-Shah2} and for sharing many insights, when this paper was coming to finalise. We thank Victor Wang for asking questions that lead this article to the current form. We are grateful to Tim Browning and Roger Heath-Brown for their interests and comments. We also thank Cécile Dartyge, Emmanuel Kowalski, Han Wu, Shucheng Yu for helpful discussions.


\begin{thebibliography}{99}
		\bibitem{Borovoi-Rudnick}
		M. Borovoi \and Z. Rudnick, Hardy–Littlewood varieties and semisimple groups. \emph{Invent. Math.} \textbf{119} (1995) 37–66.
		\bibitem{CL-T} A. Chambert-Loir \and Yu. Tschinkel, Igusa integrals and volume asymptotics in analytic and adelic geometry. \emph{Confluentes Math.}, \textbf{2} (3):351–429, 2010.
		\bibitem{CT-XuCompositio} J.-L. Colliot-Thélène \and F. Xu, Brauer-Manin obstruction for integral points of homogeneous spaces and representation of integral quadratic forms. \emph{Compositio Math.} \textbf{145}, 2009, 309–363.
		\bibitem{CT-Xu} J.-L. Colliot-Thélène \and F. Xu, Strong approximation for the total space of certain quadric fibrations   \emph{Acta Arith.} \textbf{157} (2013) 169-199.
		\bibitem{Dartyge-Martin} C. Dartyge \and G. Martin, Exponential sums with reducible polynomials, \emph{Discrete Analysis} \textbf{15} (2019), 31 pp.
		  \bibitem{D-F-I} H. Iwaniec;  W. Duke \and J. Frielander, Bounds for automorphic $L$-functions. \emph{Invent. Math.} \textbf{112} (1993): 1-8.
		\bibitem{D-F-I2} W. Duke; J. Friedlander \and H. Iwaniec, Weyl sums for quadratic roots, \emph{Int. Math. Res. Notices} (2012) \textbf{2012} no. 11, 2493-2549.
 		\bibitem{Duke-Rudnick-Sarnak} W. Duke; Z. Rudnick \and P. Sarnak, Density of integer points on affine homogeneous varieties.  \emph{Duke Math. J.} \textbf{71} (1993), 143–179.
		\bibitem{EM} A. Eskin \and C. McMullen, Mixing, counting, and equidistribution in Lie groups, \emph{Duke Math. J.} \textbf{71} (1993), 181-209.
		\bibitem{F-I} J.B. Friedlander \and H. Iwaniec, Small Representations by Indefinite Ternary Quadratic Forms. In  \emph{Number Theory and Related Fields: In Memory of Alf van der Poorten}, Springer Proceedings in Mathematics \& Statistics \textbf{43}, J.M. Borwein et al. (eds.), 157--164.
		\bibitem{H-Bdelta}  D.R. Heath-Brown, A new form of the circle method, and its application to quadratic forms. {\em J. Reine Angew. Math.} {\bf 481} (1996), 149--206.
		\bibitem{HuangII} Z. Huang. Quantitative strong approximation for ternary quadratic forms II, \emph{preprint}.
		\bibitem{HuangIII} Z. Huang. Quantitative strong approximation for ternary quadratic forms III, \emph{preprint}.
		\bibitem{PART1} Z. Huang; D. Schindler \and A. Shute, Quantitative optimal weak approximation for projective quadrics, \emph{preprint}.
		\bibitem{PART2} Z. Huang; D. Schindler \and A. Shute, Quantitative strong approximation for quaternary quadratic forms, \emph{preprint}.
		\bibitem{HooleyActa} C. Hooley, On the number of divisors of quadratic polynomials. \emph{Acta Math.} \textbf{110} (1963) (1) 97–114.
		\bibitem{Hooley} C. Hooley, On the distribution of the roots of polynomial congruences, \emph{Mathematika} \textbf{11} (1964) 39–49.
		\bibitem{HKKL} T. Hulse, C. Kuan, E. Kıral \and L.-M. Lim, Counting square discriminants, \emph{J. Number Theory} \textbf{162} (2016) 255–274.
		\bibitem{Huxley} M. N. Huxley, A note on polynomial congruences, in \emph{Recent Progress in Analytic NumberTheory} \textbf{1}, (Durham, 1979), Academic Press, New York, 1981, pp. 193–196.
		\bibitem{Iwaniec-Kolwalski}	H. Iwaniec \and E. Kowalski, \emph{Analytic number theory}. American Mathematical Society Colloquium Publications, \textbf{53}. American Mathematical Society, Providence, RI, 2004. xii+615 pp.
		\bibitem{K-K} D. Kelmer \and A. Kontorovich, Effective equidistribution of shears and applications, \emph{Math. Ann.} (2018) \textbf{370}, 381–421.
		\bibitem{K-K2} D. Kelmer \and A. Kontorovich, Exponents for the equidistribution of shears and applications. \emph{J. Number Theory}. \textbf{208} (2020) 1–46.
		\bibitem{Nagell} N. Nagell, \emph{Introduction to Number Theory.} John Wiley \& Sons, Inc., New York; Almqvist \& Wiksell, Stockholm, 1951. 309 pp.
		\bibitem{Oh-Shah} H. Oh \and N.A. Shah, Limits of translates of divergent geodesics and integral points on one-sheeted hyperboloids, {\em Israel J. Math.} \textbf{199} no. 2 (2014), 915–931.
		\bibitem{Oh-Shah2} H. Oh \and N.A. Shah, Effective orbital counting with finite stabilizers, \emph{preprint}.
		\bibitem{Tenenbaum} G. Tenenbaum, \emph{Introduction à la théorie analytique et probabiliste des nombres,} Cours Spécialisés \textbf{1}, Soc. Math. France, Paris, 1995.
		\bibitem{Weil} A. Weil, On some exponential sums. \emph{Proc. Nat. Acad. Sci. USA} \textbf{34} (1948), 204--207.
        \bibitem{Xu-Zhang} F. Xu \and R. Zhang, Counting integral points on indefinite ternary quadratic equations over number fields. {\em The Quarterly Journal of Mathematics}, Volume \textbf{74}, Issue 2 (2023) 659–685.
	\end{thebibliography}
\end{document}